\documentclass[11pt,a4paper,final, dvipsnames]{amsart}

\input xy
\xyoption{all}

\DeclareMathAlphabet{\mathpzc}{OT1}{pzc}{m}{it}

\usepackage{latexsym,amsmath,amsfonts,amscd,amssymb}
\usepackage{graphics,color}

\usepackage[all]{xy}
\usepackage{graphicx}
\usepackage{stmaryrd}
\usepackage{enumitem}
\usepackage{comment}

\usepackage{empheq}
\usepackage{color}
\usepackage{xcolor}
\usepackage{bm}

\theoremstyle{plain}
\newtheorem{theorem}{Theorem}[section]

\newtheorem{lemma}[theorem]{Lemma}
\newtheorem{proposition}[theorem]{Proposition}
\theoremstyle{definition}
\newtheorem{definition}[theorem]{Definition}
\newtheorem{remark}[theorem]{Remark}
\newtheorem{example}[theorem]{Example}
\newtheorem*{theorem*}{Theorem}

\newcommand{\red}{\mathrm{red}}
\newcommand{\irr}{\mathrm{irr}}

\newcommand{\g}{\gamma}

\newcommand{\G}{\Gamma}

\newcommand{\cA}{\mathcal{A}}

\newcommand{\cD}{\mathcal{D}}

\newcommand{\cL}{\mathcal{L}}

\newcommand{\RR}{\mathbb{R}}
\newcommand{\CC}{\mathbb{C}}

\newcommand{\PP}{\mathbb{P}}

\newcommand{\ZZ}{\mathbb{Z}}

\newcommand{\x}{\times}
\newcommand{\ox}{\otimes}

\newcommand{\la}{\langle}
\newcommand{\ra}{\rangle}
\newcommand{\frM}{\mathfrak{M}}

\newcommand{\frX}{\mathfrak{X}}

\newcommand{\Id}{\mathrm{Id}}
\newcommand{\id}{\mathrm{Id}}

\newcommand{\Sym}{\mathrm{Sym}}

\newcommand{\floor}[1]{\left \lfloor #1 \right \rfloor}

\DeclareMathOperator{\Hom}{Hom}

\DeclareMathOperator{\Gr}{Gr}
\DeclareMathOperator{\GL}{GL}

\DeclareMathOperator{\PGL}{PGL}
\DeclareMathOperator{\SL}{SL}
\DeclareMathOperator{\SU}{SU}

\DeclareMathOperator{\Stab}{Stab}
\DeclareMathOperator{\Spec}{Spec}
\DeclareMathOperator{\tr}{tr}

\makeatletter
\@namedef{subjclassname@2020}{%
  \textup{2020} Mathematics Subject Classification}
\makeatother

 \title[Representation varieties of twisted Hopf links]{Representation varieties of twisted Hopf links}
 
 
 \author[A. Gonz\'alez-Prieto]{\'Angel Gonz\'alez-Prieto}
\address{Departamento de \'Algebra, Geometr\'ia y Topolog\'ia, Facultad de Ciencias Matem\'aticas, Universidad Complutense de Madrid, 28040 Madrid, Spain.}
\address{Instituto de Ciencias Matem\'aticas (CSIC-UAM-UC3M-UCM), C.\ Nicol\'as Cabrera 15, 28049 Madrid, Spain}
\email{angelgonzalezprieto@ucm.es}
  
 \author[V. Mu\~{n}oz]{Vicente Mu\~{n}oz}
\address{Departamento de \'Algebra, Geometr\'ia y Topolog\'ia, Facultad de Ciencias Matem\'aticas, Universidad Complutense de Madrid, 28040 Madrid, Spain.}
 \email{vicente.munoz@ucm.es}

\usepackage{amsmath}

\begin{document}

\begin{abstract}
In this paper, we study the representation theory of the fundamental group of the complement of a Hopf link with $n$ twists. A general
framework is described to analyze the $\SL_r(\CC)$-representation varieties of these twisted Hopf links as byproduct of a combinatorial
problem and equivariant Hodge theory. As application, close formulas of their $E$-polynomials are provided for ranks $2$ and $3$, both for the
representation and character varieties.
\end{abstract}

\maketitle

\section{Introduction}\label{sec:introduction}

This work studies a special type of algebraic invariants of $3$-dimensional links. To be precise, given a link $L \subset S^3$ and a complex affine algebraic group $G$, we can form the so-called $G$-representation variety of the link
$$
	R(L, G) = \Hom(\pi_1(S^3-L), G),
$$
which parametrizes representations of the fundamental group of the link complement into $G$. This set can be naturally equipped with an algebraic structure in such a way that $R(L, G)$ becomes a complex affine variety. In particular, its cohomology is endowed with a mixed Hodge structure from which we can compute the $E$-polynomial
$$
	e(R(L,G)) = \sum_{k,p,q} (-1)^k h_c^{k,p,q}(R(L,G)) \, u^pv^q \in \ZZ[u,v],
$$
where $h^{k,p,q}_{c}(R(L,G))= h^{p,q}(H_{c}^k(R(L,G)))=\dim \Gr^{p}_{F}\Gr^{W}_{p+q}H^{k}_{c}(R(L,G))$ are the compactly supported Hodge numbers of $R(L,G)$. In the case that $h_c^{k,p,q}(R(L,G)) = 0$ for $p \neq q$, it is customary to write the $E$-polynomial in the variable $q = uv$.

Since the fundamental group of the link complement does not vary under diffeotopy of the link, the $E$-polynomial $e(R(L,G))$ is an algebraic invariant of the link $L$ up to link equivalence. This $E$-polynomial provides an invariant encoding the algebraic structure of the representation variety attached to $L$, and typically differs from other classical invariants of $L$ such as its Jones polynomial \cite{kauffman} or its $A$-polynomial \cite{CCGLS}. In fact, the geometry of the representation variety has been exploited several times in the literature to prove striking results of $3$-manifolds. 
For instance, in the foundational work of Culler and Shalen \cite{CS}, the authors used some simple properties of the $\SL_2(\CC)$-representation variety to provide new proofs of Thurston's theorem stating that the space of hyperbolic structures on an acylindrical $3$-manifold is compact, 
and of the Smith conjecture, which claims that any quotient with cyclic stabilizers of a closed oriented $3$-manifold with non-trivial branch knot is not simply-connected \cite[Section 5]{CS}.

Representation varieties also play a central role in mathematical physics. In the very influential paper \cite{witten1989quantum}, Witten applied Chern-Simons theory to geometrically quantize $\SU(2)$-representation varieties of knot complements, leading to a Topological Quantum Field Theory that computes the Jones polynomial of the knot. In some sense, our approach of looking at the $E$-polynomial of the representation variety can be understood as an alternative quantization of the representation varieties, more similar to 
Fourier-Mukai transforms in derived geometry \cite{Huy}, in the sense that it consists of a pull-push construction (with identity kernel), than to path integrals as arising in Chern-Simons theory \cite{witten1989quantum} (see \cite{GPLM-2017} for more information).

For these reasons, the computation of the 
$E$-polynomials $e(R(L,G))$ has been object of intense research in the recent years. 
The representation variety of torus knots for $G = \SL_2(\CC)$ was studied in \cite{Munoz}, and for $G=\SU(2)$ in \cite{Martinez-Munoz:2015}; 
whereas the $G=\SL_3(\CC)$ case was accomplished in \cite{MP}, and recently the case $G=\SL_4(\CC)$ in \cite{gonzalez2020motive} through a computer-aided proof. More exotic knots have also been studied, as the figure eight knot in \cite{HMP}. 
However, despite of these advances for representation varieties of knots, almost nothing is known in the case of links. The most studied case is the character variety of trivial links, i.e.\ representations of the free group, addressed in works such as \cite{florentino2009topology,Florentino-Lawton:2012,Lawton} (focused on the topology) and \cite{cavazos2014polynomial,Florentino-Nozad-Zamora:2019,LM} (computing the $E$-polynomials). Very recently, more complicated links were studied, such as the twisted Alexander polynomial for the Borromean link in \cite{ChenYu}.

The aim of this work is to give the first steps towards an extension of the techniques to links. In particular, we shall focus
on the ``twisted'' Hopf link $H_n$, obtained by twisting a classical Hopf link with $2$ crossings to get $2n$ crossings, as depicted in Figure \ref{fig0}.

\begin{figure}[ht]
\begin{center}
\includegraphics[width=4.5cm]{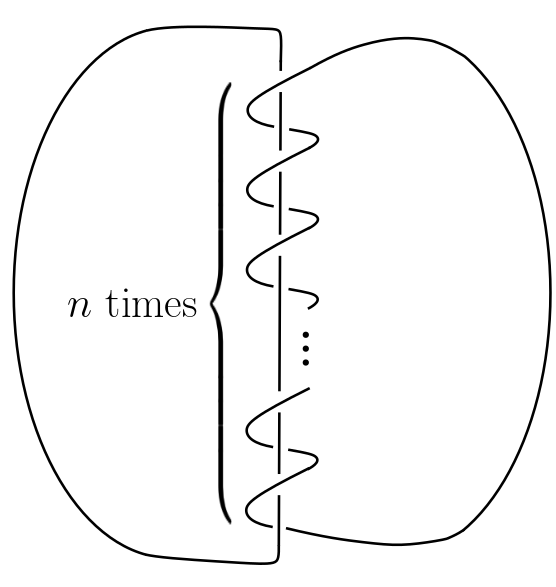}
\caption{\label{fig0} The twisted Hopf link of $n$ twists.}
\end{center}
\end{figure}

The fundamental group of the link complement of $H_n$ can be computed through a Wirtinger presentation (Proposition \ref{prop:fund-group-link}) giving rise to the group $\Gamma_n = \langle a, b \,|\, [a^n,b] = 1 \rangle$. Therefore, the associated $G$-representation variety is
$$
	R(H_n, G) = \left\{(A,B) \in G^2 \,|\, [A^n,B] = 1\right\}.
$$
In this sense, $R(H_n, G)$ should be understood as the variety counting ``supercommuting'' elements of $G$, generalizing the case $n=1$ of the usual Hopf link that corresponds to commuting elements, as studied in \cite{FlorentinoSilva,LMN}.

One of the main challenges we face in the study of the geometry of $R(H_n, G)$ is the analysis of the map $p_n: G \to G$, $A \mapsto A^n$. In this paper, we propose to split this analysis into two different frameworks, that we call the \emph{combinatorial} and the 
\emph{geometric}. The combinatorial 
setting focuses on the study of the configuration space of possible eigenvalues, and how it can degenerate under the map $p_n$. We will show in Section \ref{sec:combinatorial-setting} that understanding these degenerations can be done systematically, 
and eventually it is performed by means of a thorough application of the inclusion-exclusion principle.

The geometric setting is discussed in 
Section \ref{sec:geometric-setting}, where we show how the $E$-polynomial of the representation 
variety can be obtained from the possible Jordan forms. To this aim, 
both the stabilizer of $A^n$ in $\SL_r(\CC)$ under the conjugacy action (to parametrize the possible matrices $B$) 
and the stabilizer of $A$ (through the conjugacy orbit of the Jordan form) play a role.
Moreover, in the cases in which the 
Jordan form is not unique, but only unique up to permutation of eigenvalues, 
we show how the quotient by the corresponding symmetric group can be computed 
via equivariant Hodge theory, as developed in Section \ref{sec:equivariant-e-pol}.

To show the feasibility of this approach, we apply it to the cases of rank $2$ (Section \ref{sec:rank2}) and 
rank $3$ (Sections \ref{sec:rank3-combinatorial} and \ref{sec:rank3-geometric}), obtaining the main result of this paper.

\begin{theorem*}
The $E$-polynomials of the $\SL_r(\CC)$-representation variety of the twisted Hopf link $H_n$ with $n$ twists for ranks $r = 2,3$, are the following.
\begin{align*}
	e\big(R(H_n, & \SL_{2}(\CC))\big) =  \left( (n-1)q^2+ n q -n+5)\right)(q^3-q), \\
e\big(R(H_n, & \SL_{3}(\CC))\big) =\,   (q^3-1)(q^2-1)q^2 \Big( \floor{\frac{n}{2}} (q^2-q)(q^2-q-1) \\
 & +\frac12n^2 (q^7+2q^6+2q^5+q^4-3q^3-3q^2+2q) 
 -\frac12n (3q^7+6q^6-3q^4 -17q^3 \\ & \qquad -q^2+12q)  
 + q^7+2q^6-q^5-2q^4-6q^3+2q^2+13q\Big).
\end{align*}
\end{theorem*}

Additionally, in this paper we will go a step forward and also study the associated character varieties. 
The key point is that, if we want to obtain a genuine moduli space, we must identify isomorphic representations.
This can be done by means of the GIT quotient of the representation variety $R(L,G)$ 
under the adjoint action of $G$, giving rise to the so-called character variety
 $$
 \frM(L,G) = R(L, G) \sslash G.
 $$

It is well known \cite{LuMa} that every representation is equivalent, under the GIT quotient, to a semi-simple representation.
The semi-simple representations are those that are direct sums of irreducible ones. Hence $\frM(L,G)$ is stratified
according to partitions of $r$, where $G=\SL_r(\CC)$, corresponding to representations that are sums of irreducible 
representations of the ranks given in the partition. The $E$-polynomial of the reducible locus $\frM^{\red}(L,G)$ is computed
inductively from the irreducible representations of lower ranks. 

In the case of the twisted Hopf link $H_n$, 
to compute the $E$-polynomial $e(\frM^{\irr}(H_n,G))$ we use the characterization that 
a representation is irreducible when $A,B$ do not both leave invariant a proper subspace.
This strategy is accomplished for rank $2$ (Section \ref{sec:rank2-character}) and  rank $3$ (Section \ref{sec:rank3-character}), 
leading to the following result.

\begin{theorem*}
The $E$-polynomials of the $\SL_r(\CC)$-character variety of the twisted Hopf link $H_n$ with $n$ twists for ranks $r = 2,3$, are the following.
\begin{align*}
	e\left(\frM(H_n, \SL_{2}(\CC))\right) &= q^2+1+(n-1)(q^2-q+1), \\
e(\frM(H_n, \SL_3(\CC))) &= q^4 + q^2 + 1 + \frac{1}{2}(n^2 - 3n + 2)\left(q^6 + 2q^5 - 4q^4 + q^3 + 3q^2 - 3q + 2\right) \\
& \qquad + 3(n - 1)(q^4 - q^3 + q^2 - q + 1) + (n - 1)(q-1)\left(q^3 - 2q^2 + q \right) \\
& \qquad -  \floor{\frac{n-1}2} (q^3 - 2q^2 + 1)(q-1).
\end{align*}
\end{theorem*}

It is worth mentioning that the strategies of computation described in this paper are not restricted to low rank, and work verbatim for arbitrary rank. However, the combinatorial analysis becomes exponentially more involved with increasing rank, so the higher rank cases are untreatable with a direct counting. An interesting future work would be to algorithmize the procedure of solving the combinatorial problem, so that the higher rank cases could be addressed via a computer aided-proof, as done in \cite{gonzalez2020motive} for torus knots.

Finally, we would like to point out that this work is cornerstone to the understanding of representation varieties of general 
$3$-manifolds. Recall that the Lickorish-Wallace theorem \cite{lickorish1962representation} states that any closed orientable 
connected $3$-manifold can be obtained by applying Dehn surgery around a link 
$L \subset S^3$. This highlights the importance of (i) studying representation varieties for 
general links, not only knots; and (ii) the key role that the maps $p_n(A)=A^n$ play in this project, since they appear as part of the automorphism of the fundamental group of the torus around which surgery takes place.

\noindent \textbf{Acknowledgements.} 
We are grateful to Joan Porti for useful correspondence. The authors greatly thank Javier Mart\'inez for fixing a mistake in the calculation of the $E$-polynomial of $\frM(H_n, \SL_3(\CC))$ in a previous version of this manuscript.
The first author is partially supported by Project MCI (Spain) 
PID2019-106493RB-I00 and the second author is partially supported by Project MCI (Spain) PID2020-118452GB-I00.

\section{Representation varieties and character varieties}\label{sec:character}

Let $\G$ be a finitely generated group, and let $G$ be a complex reductive Lie group. 
A \textit{representation} of $\G$ in $G$ is a homomorphism $\rho: \G\to G$.
Consider a presentation $\G=\la \g_1,\ldots, \g_k | \{r_{\lambda}\}_{\lambda \in \Lambda} \ra$, 
where $\Lambda$ is the (possibly infinite) indexing set of relations of $\G$. Then $\rho$ is completely
determined by the $k$-tuple $(A_1,\ldots, A_k)=(\rho(\g_1),\ldots, \rho(\g_k))$
subject to the relations $r_\lambda(A_1,\ldots, A_k)=\id$, for all $\lambda \in \Lambda$. The 
\emph{representation variety} is
 \begin{align}\label{eq:rep-var}
 \begin{split}
 R(\G,G) &=\,  \Hom(\G, G) \\
  &=\,  \{(A_1,\ldots, A_k) \in G^k \, | \,
 r_\lambda(A_1,\ldots, A_k)=\id ,\forall \lambda\, \}\subset G^{k}\, .
 \end{split}
 \end{align}
Therefore $R(\G,G)$ is an affine algebraic set. 
Even though $\Lambda$ may be an infinite set, $ R(\G,G)$ is defined by finitely many equations, as a consequence of
the noetherianity of the coordinate ring of $G^k$.

We say that two representations $\rho$ and $\rho'$ are
equivalent if there exists $g \in G$ such that $\rho'(\gamma)=g^{-1} \rho(\gamma) g$,
for every $\gamma \in \G$. 
The moduli space of representations, also known as the \emph{character variety}, is the GIT quotient
 $$
 \frM(\G,G) = R(\G,G) \sslash G \, .
 $$
Recall that by definition of the GIT quotient for an affine variety, if we write
$ R(\G,G)=\Spec A$, then $M (\G,G)=\Spec A^{G}$, where $A^G$ is the finitely generated $k$-algebra of invariant elements
of $A$ under the induced action of $G$.

A representation $\rho$ is \textit{reducible} if there exists some proper linear
subspace $W\subset V$ such that for all $\g \in \G$ we have 
$\rho(\g)(W)\subset W$; otherwise $\rho$ is
\textit{irreducible}. 
If $\rho$ is reducible, then there is a flag of subspaces $0=W_0\subsetneq W_1\subsetneq \ldots \subsetneq W_r=V$
such that $\rho$ leaves $W_i$ invariant, and it induces an irreducible representation $\rho_i$ in the quotient
$V_i=W_i/W_{i-1}$, $i=1,\ldots,r$. Then $\rho$ and $\hat \rho=\bigoplus \rho_i$
define the same point in the quotient $\frM(\G,G)$. We say that $\hat\rho$ is a semi-simple
representation, and that $\rho$ and $\hat\rho$ are 
S-equivalent. The space $\frM(\G,G)$ parametrizes semi-simple representations
\cite[Thm.~ 1.28]{LuMa} up to conjugation. 

The name `character variety' for $ \frM(\G,G)$ is justified by the following fact. Suppose now that $G=\SL_r(\CC)$. Given a representation $\rho: \G\to G$, we define its
\textit{character} as the map $\chi_\rho: \G\to \CC$,
$\chi_\rho(g)=\tr \rho (g)$. Note that two equivalent
representations $\rho$ and $\rho'$ have the same character.
There is a character map $\chi: R(\G,G)\to \CC^\G$, $\rho\mapsto
\chi_\rho$, whose image
 $$
 \frX(\G,G)=\chi(R(\G,G))
 $$
leads to a natural algebraic map  
 \begin{equation}\label{eq:map-character}
 \frM(\G,G)\to \frX(\G,G).
 \end{equation}
It turns out that this map is an  isomorphism for $G = \SL_{r}(\CC)$ cf.\ \cite[Chapter 1]{LuMa}. This is the same as to say that 
$A^G$ is generated by the traces $\chi_\rho$, $\rho\in R(\G,G)$. In other words, in this case $ \frM(\G,G)$ is made of characters, justifying its name. However, for other reductive groups the map (\ref{eq:map-character}) may not be an isomorphism, as for $G = \mathrm{SO}_2(\CC)$ \cite[Appendix A]{Florentino-Lawton:2012}. For a general discussion on this issue, see \cite{Lawton-Sikora:2019}.
 
\subsection{Hodge structures and $E$-polynomials} \label{subsec:e-poly}

A pure Hodge structure of weight $k$ consists of a finite dimensional complex vector space
$H$ with a real structure, and a decomposition $H=\bigoplus_{k=p+q} H^{p,q}$
such that $H^{q,p}=\overline{H^{p,q}}$, the bar meaning complex conjugation on $H$.
A Hodge structure of weight $k$ gives rise to the so-called Hodge filtration, which is a descending filtration
$F^{p}=\bigoplus_{s\ge p}H^{s,k-s}$. We define $\Gr^{p}_{F}(H):=F^{p}/ F^{p+1}=H^{p,k-p}$.

A mixed Hodge structure consists of a finite dimensional complex vector space $H$ with a real structure,
an ascending (weight) filtration $\cdots \subset W_{k-1}\subset W_k \subset \cdots \subset H$
(defined over $\RR$) and a descending (Hodge) filtration $F$ such that $F$ induces a pure Hodge structure of weight $k$ on each $\Gr^{W}_{k}(H)=W_{k}/W_{k-1}$. We define $H^{p,q}:= \Gr^{p}_{F}\Gr^{W}_{p+q}(H)$ and write $h^{p,q}$ for the {\em Hodge number} $h^{p,q} :=\dim H^{p,q}$.

Let $Z$ be any quasi-projective algebraic variety (possibly non-smooth or non-compact). 
The cohomology groups $H^k(Z)$ and the cohomology groups with compact support  
$H^k_c(Z)$ are endowed with mixed Hodge structures \cite{De}. 
We define the {\em Hodge numbers} of $Z$ by
$h^{k,p,q}_{c}(Z)= h^{p,q}(H_{c}^k(Z))=\dim \Gr^{p}_{F}\Gr^{W}_{p+q}H^{k}_{c}(Z)$ .
The $E$-polynomial is defined as 
 $$
 e(Z):=\sum _{p,q,k} (-1)^{k}h^{k,p,q}_{c}(Z) u^{p}v^{q}.
 $$

The key property of Hodge-Deligne polynomials that permits their calculation is that they are additive for
stratifications of $Z$. If $Z$ is a complex algebraic variety and
$Z=\bigsqcup_{i=1}^{n}Z_{i}$, where all $Z_i$ are locally closed in $Z$, then $e(Z)=\sum_{i=1}^{n}e(Z_{i})$.
Also $e(X\x Y)=e(X)e(Y)$ or, more generally, $e(X) = e(F)e(B)$ for any fiber bundle $F \to X \to B$ in the Zariski topology \cite[Proposition 4.6]{GP-2018b}.
Moreover, by \cite[Remark 2.5]{LMN} if $G\to X\to B$ is a principal fiber bundle with $G$ a 
connected algebraic group, then $e(X)=e(G)e(B)$.

When $h_c^{k,p,q}=0$ for $p\neq q$, the polynomial $e(Z)$ depends only on the product $uv$.
This will happen in all the cases that we shall investigate here. In this situation, it is
conventional to use the variable $q=uv$. If this happens, we say that the variety is {\it of balanced type}.
Some cases that we shall need are:
  \begin{itemize}
 \item $e(\CC^r)=q^r$.
 \item $e(\CC^*)=q-1$.
 \item $e(\GL_r(\CC))= (q^r-1)(q^r-q)\cdots (q^r-q^{r-1})$.
 \item $e(\SL_r(\CC))=e(\PGL_r(\CC)) = (q^r-1)(q^r-q) \cdots (q^r-q^{r-2})q^{r-1}$.
  \item $e(\PP^r)=q^r+\ldots+ q^2 + q+1$.
  \item $\Sym^r(\PP^1)=\PP^r$ hence $e(\Sym^r(\PP^1))=q^r+\ldots+q+1$.  
\item By \cite{GLM}, we have that $\zeta_{\PP^n}(t)=\sum_{r\geq 0} e(\Sym^r(\PP^n)) t^r$ satisfies the formula
  $$
 \zeta_{\PP^n}(t)=\prod_{i=0}^n \frac{1}{1-q^it}\, .
 $$
From this, we extract $e(\Sym^2(\PP^2))=q^4+q^3+2q^2+q+1$, and
$e(\Sym^3(\PP^2))=q^6+q^5+2q^4+2q^3+2q^2+q+1$.
  \end{itemize}

\subsection{Equivariant $E$-polynomial}\label{sec:equivariant-e-pol}

We enhance the definition of $E$-polynomial to the case where there is an action of a finite group (see \cite[Section 2]{LM}). 

\begin{definition}
Let $X$ be a complex quasi-projective variety on which a finite group $F$ acts. 
Then $F$ also acts on the cohomology $H^*_c(X)$ respecting the mixed Hodge structure. So
$[H^*_c(X)]\in R(F)$, the representation ring of $F$.  The \emph{equivariant 
$E$-polynomial} is defined as
 $$
 e_F(X)=\sum_{p,q,k}  (-1)^k [H^{k,p,q}_c(X)] \, u^pv^q \in R(F)[u,v].
 $$
\end{definition}

Note that the map $\dim : R(F)\to \ZZ$ recovers the usual $E$-polynomial as $\dim(e_F(X))=e(X)$.
Moreover, let $T$ be the trivial representation. Let $\ell$ be the number of irreducible 
representations of $F$, which coincides with the number of conjugacy classes of $F$. 
Let $T=T_1,T_2,\ldots, T_\ell$
be the irreducible representations in $R(F)$. Write $e_F(X)=  \sum_{j=1}^\ell a_j T_j$. Then
$e(X/F)= a_1$, the coefficient of $T$ in $e_F(X)$.

We need specifically the case of the symmetric group $S_r$.
For instance, for an action of $S_2$, there are two irreducible
representations $T,N$, where $T$ is the trivial representation, and $N$ is the non-trivial representation.
Then $e_{S_2}(X)=aT+bN$. Clearly  $e(X) = a+b$, $e(X/S_2) = a$. Therefore
 \begin{equation} \label{eqn:S2}
 \begin{aligned}
 &e_{S_2}(X)=aT+bN, \\
 &a=e(X/S_2), \\ 
 &b=e(X)-e(X/S_2).
 \end{aligned}
 \end{equation}
Note that if $X,X'$ are spaces with $S_2$-actions, then 
writing $e_{S_2}(X)=aT+bN$, $e_{S_2}(X')=a'T+b'N$, we have
$e_{S_2}(X\x X')= (aa'+bb') T+ (ab'+ba')N$ and so
  $e((X\x X')/S_2)=aa'+bb'$.

We shall use later also the case of the symmetric group $F=S_3$. 
Denote by $\alpha=(1,2,3)$ the $3$-cycle and $\tau=(1,2)$ a
transposition. There are three irreducible representations
$T,S,D$, where $T$ is the trivial one, $S$ is the sign representation,  and $D$ is the standard rerpresentation.
The sign representation is one-dimensional $S=\RR$, where $\alpha\cdot x=x$ and $\tau\cdot x=-x$.
The standard representation is two-dimensional $D=\RR^2=\CC$, where $\tau\cdot z=\overline{z}$, $\alpha\cdot z=
e^{2\pi i/3} z$.
The multiplicative table of $R(S_3)$ is easily checked to be given by 
 \begin{align*}
  T\otimes T &=T,  &&   T\otimes S =S, \\
  T\otimes D &=D, &&  S\otimes S =T, \\
  S\otimes D &=D, &&  D\otimes D =T +S + D.
 \end{align*}

Let $X$ be a variety with an $S_3$-action. Then $e_{S_3}(X) =a T+bS + cD$. Then $e(X)=a+b+2c$ and 
$e(X/S_3)=a$. For the transposition $\tau=(1,2)\in S_3$, we have $T^\tau=\RR$, $S^\tau=0$ and $D^\tau=\RR$.
Thus $e(X/\la\tau\ra)=a+c$. This implies that
 \begin{equation} \label{eqn:S3}
 \begin{aligned}
 &e_{S_3}(X) =a T+bS + cD , \\
  &a = e(X/S_3), \\ 
 &b= e(X)-2 e(X/\la\tau\ra)+e(X/S_3), \\
 &c = e(X/\la\tau\ra)- e(X/S_3).  
 \end{aligned}
 \end{equation}

An interesting case that we will also apply in Section \ref{sec:rank3-geometric} is the following.

\begin{proposition}\label{prop:lemma-connected-group}
Let $G$ be a complex algebraic group equipped with an action of a finite group $\rho: F \to \textup{Inn}(G)$ acting by inner automorphisms. If $G$ is
connected, then
$$
	e_{F}(G) = e(G)T.
$$ 
\end{proposition}

\begin{proof}
Since $\textup{Inn}(G) = G/Z(G)$, if $G$ is connected then $\textup{Inn}(G)$ also is so. Hence, any inner automorphism is connected to the identity
through a path, meaning that any inner automorphism is homotopic to the identity. Then, for all $\tau \in F$, the map $\tau \cdot : G \to G$ is null-homotopic,
so it induces a trivial action in cohomology.
\end{proof}

 \section{Twisted Hopf links}

In this paper, we shall focus on the \emph{twisted Hopf link}, which is the link formed by two circles knotted as the Hopf link but with $n$ twists, as depicted in Figure \ref{fig1}. We will denote this knot by $H_n$. 
Notice that the link $H_1$ is the usual Hopf link.
 
\begin{figure}[ht]
\begin{center}
\includegraphics[width=6cm]{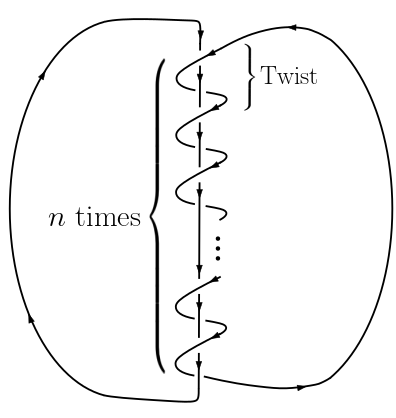}
\caption{\label{fig1} The twisted Hopf link of $n$ twists with oriented strands.}
\end{center}
\end{figure}

 \begin{proposition}\label{prop:fund-group-link}
 The fundamental group of the (complement of the) twisted Hopf link with $n \geq 1$ twists is
 $$
 \G_n:=\pi_1(S^3-H_n) = \langle a, b \,|\, [a^n,b] = 1 \rangle,
 $$
 where $[a^n,b] = a^nba^{-n}b^{-1}$ is the group commutator.
 \end{proposition}
 
 \begin{proof}
Let us compute the Wirtinger presentation of the fundamental group of $H_n$,
which is a presentation of the fundamental group of the complement of a knot that can be obtained algorithmically from the crossings of a planar representation of the knot \cite[Chapter III.D]{Rolfsen}. Let us orient the two strands of $H_n$ as shown in Figure \ref{fig1}.
From these orientation, we observe that $\pi_1(S^3-H_n)$ is generated 
by $2n$ elements, namely $x_1, y_1, \ldots, x_n, y_n$, corresponding to the $2n$ arcs of overpassing strands. For the relations, 
the knot has $n$ double crossings of the form of Figure \ref{fig2}.

\begin{figure}[ht]
\begin{center}
\vspace{-0.5cm}
\includegraphics[width=3cm]{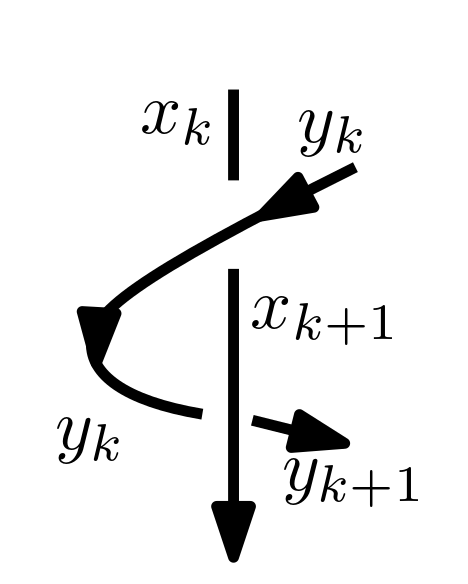}
\vspace{-0.5cm}
\caption{\label{fig2} A crossing of the $H_n$.}
\end{center}
\end{figure}

From each of these crossings, we obtain two relations $y_kx_k = x_{k+1}y_k$ and $x_{k+1}y_k = y_{k+1}x_{k+1}$ for $k = 1, 2, \ldots, n$, 
writing $x_{n+1} = x_1$ and $y_{n+1} = y_1$. Therefore, we get that
$$
\pi_1(S^3-H_n) = \langle x_1, y_1, \ldots, x_n, y_n \,|\, y_kx_k = x_{k+1}y_k, x_{k+1}y_k = y_{k+1}x_{k+1} \textrm{ for } 1\leq k \leq n\rangle.
$$
From this relations, we can solve for $y_k$ and $x_k$, for $k\geq 2$, from $x_1$ and $y_1$. Thus, the group can be also written as
$$
\pi_1(S^3-H_n) \cong \langle x_1, y_1 \,|\, (x_1y_1)^n = (y_1x_1)^n\rangle.
$$
Making the change $a = x_1y_1$ and $b = y_1$ we get the desired presentation.
 \end{proof}
 
 \begin{remark}
 For $n = 1$, the group $\pi_1(S^3-H_1) = \ZZ \times \ZZ$ coincides with the fundamental group of the $2$-dimensional torus, which is generated by two commuting elements. 
In some sense, $\pi_1(S^3-H_n)$ generalizes this result by considering `supercommutation' relations instead, of the form $a^nb=ba^n$. 
 \end{remark}

Using the description (\ref{eq:rep-var}) we directly get from Proposition \ref{prop:fund-group-link} that the $G$-representation variety of the twisted Hopf link with $n$ twists is
\begin{equation}\label{eq:rep-var-link}
	R(\Gamma_n, G) = \left\{(A, B) \in G^2 \,|\, [A^n,B] = 1\right\}.
\end{equation}

To emphasize the role of the twisted Hopf link in the representation variety, throughout this paper we shall denote $R(H_n, G) = R(\Gamma_n, G)$.

\section{The $\SL_{r}(\CC)$-representation variety of the twisted Hopf link} \label{sec:previous}

\subsection{The combinatorial setting}\label{sec:combinatorial-setting}

Given $r \geq 1$, let us consider the space of possible eigenvalues of a matrix of $\SL_{r}(\CC)$,
$$
	\Delta^r = \{(\alpha_1, \ldots, \alpha_r) \in (\CC^*)^r \,|\, \alpha_1\cdots \alpha_r = 1\}.
$$
We can see $\Delta^r$ as a (coarse) configuration space, which is naturally stratified by equalities $\alpha_i=\alpha_j$. 
Here, we only need to consider a simple case of the Fulton-MacPherson stratification. Given an 
equivalence relation $\sigma$ on $\{1, \ldots, r\}$ (equivalently, a partition of the set $\{1, \ldots, r\}$), 
let us denote by $\Delta^r_\sigma \subset \Delta^r$ the collection of $(\alpha_1, \ldots, \alpha_r)$ 
such that $\alpha_i = \alpha_j$ if and only if $i \sim_\sigma j$. Observe that if $\sigma = \{\upsilon_1, \ldots, \upsilon_s\}$, 
then there is a natural action of the group $S_\sigma := S_{t_1} \times \ldots \times S_{t_r}$ on $\Delta_\sigma^r$ 
by permutation of blocks, where $t_i$ is the number of  subsets $\upsilon_j$ of size $|\upsilon_j| = i$.

Two partitions $\sigma = \{\upsilon_1, \ldots, \upsilon_s\}$ and $\sigma' = \{\upsilon_1', \ldots, \upsilon_s'\}$
(with the same number of equivalence classes) are said to be equivalent if there exists a permutation of 
$\{1, \ldots, r\}$ sending $\upsilon_i$ to $\upsilon_i'$ for all $i$. In this manner, if we relabel the indices 
in such a way that $\upsilon_1 = \{1, \ldots, r_1\}, \upsilon_2 = \{r_1+1, \dots, r_1 + r_2\}$ and so on, 
we have a simple description
$$
	\Delta_\sigma^r = \left\{(\lambda_1, \lambda_2, \ldots, \lambda_s) \in (\CC^*)^s \,|\, 
	\lambda_1^{r_1} \lambda_2^{r_2} \ldots \lambda_s^{r_s} = 1, \lambda_i \neq \lambda_j \textrm{ for } i \neq j\right\}.
$$

\begin{example}
If $\sigma = \{\{1,2,3\}, \{4\}, \{5,6,7\}\}$, then
$$
	\Delta^5_\sigma = \{(\lambda_1, \lambda_2, \lambda_3)  \,|\, 
	\lambda_1^3\lambda_2\lambda_3^3 = 1, \lambda_1 \neq \lambda_2, \lambda_1 \neq \lambda_3, \lambda_2 \neq \lambda_3\},
$$
which is equipped with the action of $S_\sigma = S_2$ given by the map 
$(\lambda_1, \lambda_2, \lambda_3) \mapsto (\lambda_3, \lambda_2, \lambda_1)$. The partition $\sigma$ 
is equivalent, for instance, to the partition $\sigma' = \{\{1,2,3\}, \{4,5,6\}, \{7\}\}$.
\end{example}

For $n\geq 1$, there is a natural map $p_n: \Delta^r \to \Delta^r$ given by 
$p_n(\alpha_1, \ldots, \alpha_r) = (\alpha_1^n, \ldots, \alpha_r^n)$. We say that $\sigma'$ \textit{refines} $\sigma$
if 
the partition of $\sigma'$ is obtained from that of $\sigma$ by extra subdivisions. We indicate this
as $\sigma' \to \sigma$. If $\sigma'$ is a refinement of $\sigma$, let us denote 
 $$ 
 \Delta^r_{\sigma' \to \sigma} = \Delta^r_{\sigma'} \cap p_n^{-1}(\Delta^r_{\sigma}).
 $$ 

A key observation is that we 
can compute $e(\Delta^r_{\sigma' \to \sigma})$ recursively. 
For instance, for notational simplicity, 
let us suppose that $\sigma = \{\upsilon_1, \ldots, \upsilon_s\}$,
$\sigma' = \{\upsilon_1', \ldots, \upsilon_{s+t}'\}$, with
$\upsilon_{s} = \upsilon_{s}' \cup \upsilon_{s+1}' \cup \ldots \cup \upsilon_{s+t}'$ 
and $\upsilon_i = \upsilon_i'$ for $i < s$. In that case, we get that if $(\lambda_1, \ldots, \lambda_{s+t}) 
\in \Delta^r_{\sigma' \to \sigma}$, since $\lambda_{s}^n = \lambda_{s+1}^n = \ldots = \lambda_{s+t}^n$, 
we must have $\lambda_{s+k} = \lambda_s\varepsilon_k$ for all $k = 1, \ldots, t$ and some pairwise different 
roots of unit $\varepsilon_k \in \mu_{n}^*$. Here, we denote $\mu_n=\{e^{2\pi i k /n} |\, k=0,1,\ldots, n-1\}$ and $\mu_n^* = \mu_n - \{1\}$.
Hence,
$$
	\Delta^r_{\sigma' \to \sigma} = \left\{(\lambda_1, \ldots, \lambda_s, \varepsilon_1, \ldots, \varepsilon_t) \in (\CC^*)^s \times (\mu_n^*)^t \,\left|\, 
	\begin{matrix}\lambda_1^{r_1} \cdots \lambda_s^{r_s} \varepsilon_1^{r'_{s+1}} \ldots \varepsilon_t^{r'_{s+t}} = 1, \\ 
	\lambda_i \neq \lambda_j \varepsilon \textrm{ for } i \neq j, \varepsilon\in \mu_n , \\  
	\varepsilon_{k} \neq \varepsilon_{l} \textrm{ for } k \neq l 
	\end{matrix}\right.\right\}.
$$

Now, observe that if we remove any of the two later inequality conditions, we get a space of the form 
$\Delta^r_{\tilde{\sigma}' \to \tilde{\sigma}}$ for some coarser $\tilde{\sigma}'$ and $\tilde{\sigma}$. 
Therefore, we can simply compute the $E$-polynomial of the space
$$
	\left\{(\lambda_1, \ldots, \lambda_s, \varepsilon_1, \ldots, \varepsilon_t) \in (\CC^*)^s \times 
	(\mu_n^*)^t \,\left|\, \lambda_1^{r_1} \cdots \lambda_s^{r_s} \varepsilon_1^{r'_{s+1}} \ldots 
	\varepsilon_t^{r'_{s+t}} = 1\right.\right\},
$$
and then remove the strata corresponding to the different equalities using the inclusion-exclusion principle.

Finally, observe that the action of $S_{\sigma'}$ does not restrict to an action on $\Delta^r_{\sigma' \to \sigma}$. 
Instead, we find that there is a maximal subgroup $S_{\sigma' \to \sigma} < S_{\sigma'}$ acting on 
$\Delta^r_{\sigma' \to \sigma}$ namely, those permutations that preserve the refinement.

\subsection{The geometric setting} \label{sec:geometric-setting}

In this section, we show how the combinatorial set-up previously developed can be used to 
control important strata that appear in the representation varieties of Hopf links. To control the 
Jordan structure of the matrices, we need the following definition.

\begin{definition}
A \emph{Jordan type} of rank $r \geq 1$ is a tuple $\xi = (\sigma, \kappa)$ where $\sigma = \{\upsilon_1, \ldots, \upsilon_s\}$ 
is a partition of $\{1, \ldots, r\}$ and $\kappa = \{\tau_1, \ldots, \tau_s\}$ is a collection where 
$\tau_i$ is a partition of $\upsilon_i$ into \textit{linearly ordered} sets.

A type $\xi' = (\sigma', \kappa')$ is said to \textit{refine} $\xi = (\sigma, \kappa)$, and we shall denote it by $\xi' \to \xi$, 
if $\sigma'$ is a refinement of $\sigma$ and for any $\upsilon_i \in \sigma$ that decomposes as 
$\upsilon_i = \upsilon_{i_1}' \cup \ldots \cup \upsilon_{i_t}'$ in $\sigma'$ we have that $\tau_i = \tau'_{i_1} \cup \ldots \cup \tau'_{i_t}$.
\end{definition}

The rationale behind a Jordan type is that it codifies the block structure of a Jordan matrix. 
On the one hand, the partition $\sigma$ identifies the multiplicities of the eigenvalues: each of the sets $\upsilon_i$ of the partition corresponds to a collection of equal eigenvalues (each number identifies the column of the eigenvalue). On the other hand, $\kappa$ determines the inner structure of the Jordan blocks for each eigenvalue: each set of the partition $\tau_i$ corresponds to a block of the Jordan matrix associated to an eigenvalue. The total order within this set is needed to identify the eigenvector column (the last element) and the off-diagonal elements of the Jordan form: if $a$ is a succesor of
$b$, then there is a $1$ at the $(a,b)$-entry of the matrix.

\begin{example}\label{ex:jordan-types} To clarify this association, let us provide several examples.
\begin{itemize}
	\item The type $\xi_1=\big(\sigma_1=\{  \{1,2\}, \{3,4\}, \{5,6\} \} ,\tau_1=\{ \{(1,2)\}$, $\{(3), (4) \}, \{(5,6)\} \}\big)$ corresponds to Jordan matrices of the form
 $$
 A=\tiny\left( \begin{array}{cccccc} \lambda & 0 &0 &0&0& 0 \\  1 & \lambda & 0 &0 &0&0 \\ 0&0& \mu & 0 &0 & 0 \\
  0& 0&0& \mu & 0 &0  \\ 0 &0 & 0&0& \alpha & 0  \\ 0&  0&0&0&1& \alpha \end{array}  \right), 
  \qquad \normalsize \textrm{ with } \lambda \neq \mu, \lambda \neq \alpha, \mu \neq \alpha.
$$
	\item The type $\xi_2=\big(\sigma_2=\{  \{1,2, 3,4\}, \{5,6\} \} ,\tau_2=\{ \{(1,2), (3), (4) \}, \{(5,6)\} \}\big)$ corresponds
 to Jordan matrices of the form
 $$A=\tiny\left( \begin{array}{cccccc} \lambda & 0 &0 &0&0& 0 \\  1 & \lambda & 0 &0 &0&0 \\ 0&0& \lambda & 0 &0 & 0 \\
  0& 0&0& \lambda & 0 &0  \\ 0 &0 & 0&0& \alpha & 0  \\ 0&  0&0&0&1& \alpha \end{array}  \right), 
  \qquad \normalsize \textrm{ with } \lambda \neq \alpha.
$$
	\item The type $\xi_3=\big(\sigma_3=\{  \{1,5\}, \{3,4\}, \{2,6\} \} ,\tau_3=\{ \{(1,5)\}$, $\{(3), (4) \}, \{(2,6)\} \}\big)$ corresponds to Jordan matrices of the form
 $$
 A=\tiny\left( \begin{array}{cccccc} \lambda & 0 &0 &0&0& 0 \\  0 & \alpha & 0 &0 &0&0 \\ 0&0& \mu & 0 &0 & 0 \\
  0& 0&0& \mu & 0 &0  \\ 1 &0 & 0&0& \lambda & 0  \\ 0&  1&0&0&0& \alpha \end{array}  \right), 
  \qquad \normalsize \textrm{ with } \lambda \neq \mu, \lambda \neq \alpha, \mu \neq \alpha.
$$
\end{itemize}
\end{example}

There is a natural action of the symmetric group $S_r$ on the set of types by `permutation of columns': relabel each element of $\{1, \ldots, r\}$ according to the permutation both in $\sigma$ and $\tau_i$. We will say that two types are \emph{equivalent} if they lie in the same $S_r$-orbit. Notice that, by definition, in principle, for a type $\xi = (\sigma, \kappa)$, the partition $\sigma$ of $\{1,\ldots, r\}$ may not be into segments 
(i.e.\ equal eigenvalues may be sparse in the matrix). However, by `putting together' the Jordan blocks, there always exists an equivalent type for which 
$\sigma$ and each $\tau_i$ are made of segments, and the total order in $\tau_i$ agrees with the natural order in $\{1, \ldots, r\}$.

\begin{example}
The types $\xi_1$ and $\xi_3$ of Example \ref{ex:jordan-types} are equivalent under the permutation $\varphi = (2\,5)$. However, $\xi_1$ and $\xi_2$ are not equivalent.
\end{example}

Given a type $\xi$ of rank $r$, 
let us denote by $\cA_\xi$ the collection of Jordan matrices 
(with 1's as off-diagonal elements according to the total orders given by $\kappa$) of $\SL_{r}(\CC)$ whose block structure is $\xi$. 
If we allow any non-zero off-diagonal element in the entries determined by $\kappa$, we will refer to these matrices as generalized Jordan matrices.
The space of generalized Jordan matrices will be denoted $\cA^g_\xi$.
We also consider $\tilde{\cA}_{\xi} = \SL_{r}(\CC) \cdot \cA_\xi$, where the action is by conjugation.
If we pick a collection $\xi_1, \ldots, \xi_N$ of non-equivalent representatives of all the types of rank $r$, we get a decomposition
 $$
 \SL_{r}(\CC) = \tilde{\cA}_{\xi_1} \sqcup \ldots \sqcup \tilde{\cA}_{\xi_N}.
 $$

Given a type $\xi=(\sigma,\kappa)$, we have a group $S_\xi< S_\sigma$. This is the group of 
permutations of $\sigma=(\upsilon_1,\ldots,  \upsilon_s)$ that permute blocks $\upsilon_i$ whose decompositions
$\tau_i$ are equivalent (of the same number of sets and of the same sizes). The group $S_\xi$ acts on $\cA_\xi$ as follows:
let $\varphi \in S_\xi$ and $A\in  \cA_\xi$. So $\varphi$ is a permutation of the blocks $\upsilon_i$ of $\sigma$, that is, of the eigenvalues $\lambda_i$.
For such eigenvalues, the Jordan forms match exactly since the corresponding decompositions $\tau_i$ are equivalent, 
therefore the matrix $A'$ with the eigenvalues $\lambda_{\varphi(i)}$ lies in $\cA_\xi$ as well. Moreover, as this
is given by a change of the basis by permutation of the vectors, there is a matrix $P_\varphi$ such that 
$A'=P_\varphi^{-1}AP_\varphi$. This $P_\varphi$ is well defined in $\PGL_r(\CC)/\Stab(A)$ (quotient by action on the left). 
The action on $\PGL_r(\CC)/\Stab(A)$ is by product of $P_\varphi$ on the right.

Now, we consider the map
  $$
    p_n: \SL_{r}(\CC)  \to \SL_{r}(\CC) , \quad p_n(A) = A^n,
 $$
and we set 
 $$
 \begin{aligned}
  &{\cA}_{\xi' \to \xi} = \cA^g_{\xi'} \cap p_n^{-1}(\cA_{\xi}) , \\ 
  & \tilde{\cA}_{\xi' \to \xi} = \SL_{r}(\CC) \cdot {\cA}_{\xi' \to \xi}. 
  \end{aligned}
  $$

Moreover, if $\xi'$ refines $\xi$, then there is a group $S_{\xi' \to \xi}$ of permutations of $\xi'$ 
(that is, permutations $\varphi$ of $\sigma'$ that respect the blocks of $\xi'$) that induce permutations of $\xi$ (it is
enough that they induce a permutation on $\sigma$ under the refinement $\sigma'\to \sigma$).

\begin{example}
 Suppose we have types $\xi'=\big(\sigma'=\{  \{1,2\}, \{3,4\}, \{(5),(6)\} \} ,\tau'=\{ \{(1,2)\}$, $\{3,4 \}, \{(5,6)\} \}\big)$,
 $\xi=\big(\sigma=\{  \{1,2, 3,4\}, \{5,6\} \} ,\tau=\{ \{(1,2), (3), (4) \}, \{(5,6)\} \}\big)$  
 (cf.\ Example \ref{ex:jordan-types}). So $\xi'\to \xi$. This corresponds
 to matrices
 $$A=\tiny\left( \begin{array}{cccccc} \lambda & 0 &0 &0&0& 0 \\  1 & \lambda & 0 &0 &0&0 \\ 0&0& \lambda\varepsilon & 0 &0 & 0 \\
  0& 0&0& \lambda\varepsilon & 0 &0  \\ 0 &0 & 0&0& \alpha & 0  \\ 0&  0&0&0&1& \alpha \end{array}  \right), 
  \qquad \normalsize \textrm{ with } \varepsilon\in\mu^*_n \textrm{ and }\alpha \neq \lambda\epsilon, \forall \epsilon\in\mu_n.
$$
The group $S_{\xi'} =S_2$ permuting $\{1,2\}$ and $\{5,6\}$, whereas $S_{\xi'\to \xi }=\{1\}$. Note also that $S_{\sigma'}=S_3$.
\end{example}

\begin{lemma} \label{lem:A}
For any types $\xi = (\sigma, \kappa)$ and $\xi' = (\sigma', \kappa')$ we have that 
$\cA_{\xi' \to \xi} = \emptyset$ if $\xi'$ does not refine $\xi$, and 
$\cA_{\xi' \to \xi} \cong \Delta^r_{\sigma' \to \sigma}$ if it does. 
\end{lemma}

\begin{proof}
A direct computation shows that, since taking powers preserves the Jordan block structure, if $A \in \cA^g_{\xi'}$, then $A^n$ 
can only lie in strata of the form $\cA^g_{\xi}$ when $\xi'$ refines $\xi$. Hence, $p_n^{-1}(\cA_{\xi}) \cap \cA^g_{\xi'} 
= \emptyset$ if $\xi'$ does not refine $\xi$.

On the other hand, given $A \in \cA_{\xi' \to \xi}$ with $\xi'$ refining $\xi$, we observe that 
the off-diagonal elements of $A$ are fixed to coincide with the Jordan structure. In this manner, 
the only freedom we have is to choose the eigenvalues of $A$, and this is precisely $ \Delta^r_{\sigma' \to \sigma}$.
\end{proof}

Let us denote by
$$
	R(H_n, \SL_{r}(\CC))_{\xi \to \xi'} := \left\{(A, B) \in \tilde{\cA}_{\xi' \to \xi} \times \SL_{r}(\CC) \,|\, [A^n,B] = \Id \right\}.
$$
Observe that we have a natural stratification
$$
	R(H_n, \SL_{r}(\CC)) = \bigsqcup_{\xi', \xi} R(H_n,\SL_{r}(\CC))_{\xi' \to \xi} ,
$$
where each of the indices $\xi'$ and $\xi$ runs over a collection of non-equivalent representatives of all the types of rank $r$.
We denote by $\Stab(\xi')$ the stabilizer in $\PGL_{r}(\CC)$ of any matrix of $\cA_{\xi'}$, and 
$\widetilde{\Stab}(\xi)$ the stabilizer in $\SL_{r}(\CC)$ of any matrix of $\cA_{\xi}$. 
Note that the action of an element $\varphi\in S_{\xi'\to \xi}$ is by permutation of columns in either $\PGL_r(\CC)$, $\SL_r(\CC)$, 
$\PGL_{r}(\CC)/ \Stab(\xi')$ and $\widetilde{\Stab}(\xi)$.

\begin{proposition} \label{prop:lem:RG}
For any types $\xi = (\sigma, \kappa)$ and $\xi' = (\sigma', \kappa')$, we have that
$$
	R(H_n, \SL_{r}(\CC))_{\xi' \to \xi} \cong  
	\left.\left(\cA_{\xi' \to \xi} \times \big(\PGL_{r}(\CC)/ \Stab(\xi')\big) \times \widetilde{\Stab}(\xi) \right) \right/ S_{\xi' \to \xi} .
$$
\end{proposition}

\begin{proof}
Take $(A, B) \in \tilde{\cA}_{\xi' \to \xi} \times \SL_{r}(\CC)$ such that $[A^n,B] = \Id$. Then $A \in \SL_r(\CC) \cdot \cA_{\xi' \to \xi}$, that
is $A=P^{-1}A_0P$, with $A_0\in \cA_{\xi' \to \xi}=\cA^g_{\xi'} \cap p_n^{-1}(\cA_{\xi})$, i.e.\ 
$A_0\in \cA^g_{\xi'}$ and $A_0^n\in \cA_{\xi}$. The matrix $P$ is determined up to $\Stab(A_0)$. 
On the other hand, $B$ commutes with $A^n=P^{-1}A_0^n P$, hence $B_0=P B P^{-1}$ lies in $\Stab(A_0^n)=\widetilde{\Stab}(\xi)$. 

The pair $(A,B)$ is determined by $(A_0,P,B_0)$. This is not unique, the matrix $A_0$ can be changed by an equivalent
matrix $A_0'$ via an element $\varphi\in S_{\xi'}$. The associated triple is $(A_0'=P_\varphi^{-1}A_0P_\varphi, P'=P_\varphi^{-1}P, 
B_0'=P_\varphi^{-1}B_0 P_\varphi)$. In order for $B_0'$ to lie in the stabilizer of a matrix of $\cA_\xi$, it is needed that whenever two 
eigenvalues $\lambda_i,\lambda_j$ satisfy $\lambda_i^n=\lambda_j^n$, the permutation $\varphi$ moves then to eigenvalues
$\lambda_i',\lambda_j'$ such that $(\lambda_i')^n=(\lambda_j')^n$. This means that $\varphi$ respects the  partition $\sigma$,
that is, it lies in $S_{\xi'\to \xi}$.
\end{proof}

\section{Representation variety of the twisted Hopf link of rank $2$}\label{sec:rank2}

In this section, we shall compute the $E$-polynomial of the $\SL_{2}(\CC)$-representation variety 
of the twisted Hopf link $H_n$, that is, the space $R(H_n,\SL_2(\CC))$. 
For that purpose, we analyze the combinatorial and geometric settings as outlined in Section \ref{sec:previous}.

\subsection{Combinatorial setting}
In rank $2$ we have only $2$ possible partitions, up to equivalence. These are
$$
	\sigma_1 = \{\{1,2\}\}, \qquad \sigma_{2} = \{\{1\},\{2\}\}.
$$
The first one corresponds to $\SL_2(\CC)$ matrices with equal eigenvalues, and the second one with distinct eigenvalues.
Hence, we have
\begin{align*}
	e(\Delta^2_{\sigma_1}) &= e(\{\lambda \in \CC^*|\, \lambda^2 = 1\}) = e(\mu_2)=2, \\
	e(\Delta^2_{\sigma_2}) &= e(\{(\lambda_1, \lambda_2) \in (\CC^*)^2\,|\, \lambda_1\lambda_2 = 1, \lambda_1 \neq \lambda_2\}) 
	= e(\CC^* - \Delta^2_{\sigma_1}) = e(\CC^* - \mu_2)=q-3.
\end{align*}
Only the stratum $\Delta^2_{\sigma_2}$ has a non-trivial action of $S_2$, 
given by $(\lambda_1, \lambda_2) \mapsto (\lambda_2, \lambda_1)$.

For the degenerations, we have $e(\Delta^2_{\sigma_1 \to \sigma_1}) = e(\Delta^2_{\sigma_1}) = e(\mu_2) = 2$. In addition,
\begin{align*}
	e\left(\Delta^2_{\sigma_2 \to \sigma_1}\right) &= e\left(\{(\lambda_1, \lambda_2) \in (\CC^*)^2\,|
	\, \lambda_1\lambda_2 = 1, \lambda_1 \neq \lambda_2, \lambda_1^n = \lambda_2^n\}\right)  \\
	&= e\left(\{(\lambda_1, \varepsilon) \in \CC^* \times \mu_n\,|\, \lambda_1^2\varepsilon= 1, \varepsilon\neq 1\}\right)  \\
	&= e\left(\{(\lambda_1, \varepsilon) \in \CC^* \times \mu_n\,|\, \lambda_1^2\varepsilon = 1\}\right) 
	- e\left(\{\lambda_1 \in \CC^* \,|\, \lambda_1^2 = 1\}\right)  \\ 
	&= e\left(\mu_{2n}\right) - e\left(\mu_2\right) = 2n-2.
\end{align*}
And 
\begin{align*}
	e\left(\Delta^2_{\sigma_2 \to \sigma_2}\right) &= e\left(\{(\lambda_1, \lambda_2) \in (\CC^*)^2\,|
	\, \lambda_1\lambda_2 = 1, \lambda_1 \neq \lambda_2, \lambda_1^n \neq \lambda_2^n\}\right)  \\
	&=  e\left(\Delta^2_{\sigma_2}\right) -  e\left(\{(\lambda_1, \lambda_2) \in (\CC^*)^2\,|
	\, \lambda_1\lambda_2 = 1, \lambda_1 \neq \lambda_2, \lambda_1^n = \lambda_2^n\}\right) \\
	&= q-3 -  e\left(\Delta^2_{\sigma_2 \to \sigma_1}\right) = q-3 - (2n-2) = q-2n-1.
\end{align*}

Now, let us analyze the quotients by $S_2$. Recall that the action of $S_2$ on $\CC^*$ given by $\lambda \mapsto \lambda^{-1}$
has quotient $\CC^*/S_2 =\CC$, given by the invariant function $s=\lambda+\lambda^{-1}$. We have that
\begin{align*}
	e(\Delta^2_{\sigma_2}/S_2) &= e((\CC^* - \mu_2)/S_2) = e(\CC-\{\pm 2\})=q-2. \\
	e\left(\Delta^2_{\sigma_2 \to \sigma_1}/S_2\right) &= e\left(\mu_{2n}/S_2\right)  
	- e\left(\Delta^2_{\sigma_1}/S_2\right) = n+1-2=n-1.\\
	e\left(\Delta^2_{\sigma_2 \to \sigma_2}/S_2\right) &=  e\left(\Delta^2_{\sigma_2}/S_2\right) 
	 -  e\left(\Delta^2_{\sigma_2 \to \sigma_1}/S_2\right)  = (q-2) - (n-1) = q-n-1.
\end{align*}
Therefore, using (\ref{eqn:S2}) we get the equivariant $E$-polynomials:
\begin{equation}\label{eqn:sigma2}
\begin{aligned}
	e_{S_2}(\Delta^2_{\sigma_2}) &= (q-2)T - N. \\
	e_{S_2}\left(\Delta^2_{\sigma_2 \to \sigma_1}\right) &= (n-1)T + (n-1)N.\\
	e_{S_2}\left(\Delta^2_{\sigma_2 \to \sigma_2}\right) &=  (q-n-1) T -n N.
\end{aligned}
\end{equation}

\subsection{Geometric setting}

From the two partitions $\sigma_1$ and $\sigma_2$, we can create three types (up to equivalence), namely
 $$
 \xi_1 = (\sigma_1, \{1,2\}), \qquad \xi_2 = (\sigma_1, \{(1,2)\}), \qquad \xi_3 = (\sigma_2, \{\{1\}, \{2\}\}).
 $$
Recall that $\xi_1$ is the type of the diagonalizable matrices with equal eigenvalues (namely, $\pm \Id$), 
$\xi_2$ is the type of the two Jordan type matrices $J_{\pm}={\tiny \left(\begin{array}{cc} \pm 1 & 0  \\ 1 &\pm1 \end{array}\right)}$, 
and $\xi_3$ is the type of diagonalizable matrices with different eigenvalues.

Denote $R_{\xi_i  \to \xi_j}= R(H_n,\SL_2(\CC))_{\xi_i \to \xi_j}$.
Taking into account that the only refinement relations are $\xi_i \to \xi_i$, $i=1,2,3$,
and $\xi_3  \to \xi_1$, we have
$$
 R(H_n,\SL_2(\CC)) = R_{\xi_1 \to \xi_1} \sqcup R_{\xi_2 \to \xi_2} \sqcup 
 R_{\xi_3 \to \xi_3} \sqcup R_{\xi_3 \to \xi_1}.
$$

Counting for each stratum we have the following:
\begin{itemize}
 \item By Lemma \ref{lem:A} and Proposition  \ref{prop:lem:RG},
 $$
   R_{\xi_1 \to \xi_1} = \Delta^2_{\sigma_1 \to \sigma_1} \x \big( \PGL_{2}(\CC)/ \Stab(\xi_1) \big) \x \widetilde{\Stab}(\xi_1).
 $$
  Now $\Stab(\xi_1) =\PGL_2(\CC)$, $\widetilde{\Stab}(\xi_1) = \SL_{2}(\CC)$ and $e(\Delta^2_{\sigma_1 \to \sigma_1})=2$, hence
  $$
   e\left(R_{\xi_1 \to \xi_1}\right) = e(\Delta^2_{\sigma_1 \to \sigma_1}) e(\SL_{2}(\CC)) =  2(q^3-q).
  $$ 
 
 \item We have also
 $$
  R_{\xi_2 \to \xi_2} = \Delta^2_{\sigma_1 \to \sigma_1}\x \big( \PGL_{2}(\CC)/ \Stab(\xi_2)\big)\x\widetilde{\Stab}(\xi_2).
  $$
  Now $\Stab(\xi_2) ={\tiny \left\{\left(\begin{array}{cc} 1 & 0 \\ a & 1\end{array}\right)\right\}} \cong \CC$ 
  and  $\widetilde{\Stab}(\xi_2) ={\tiny \left\{\left(\begin{array}{cc} \pm 1 & 0 \\ a & \pm 1\end{array}\right)\right\}} \cong \CC\x\mu_2$. Hence
  $$
   e\left(R_{\xi_2 \to \xi_2}\right) = e(\Delta^2_{\sigma_1 \to \sigma_1}) \frac{q^3-q}{q} \, 2q =  4(q^3-q).
  $$ 
 
  \item We continue with $ R_{\xi_3 \to \xi_3}=\tilde R_{\xi_3 \to \xi_3} /S_2$, with
 $$
  \tilde R_{\xi_3 \to \xi_3} = \Delta^2_{\sigma_2 \to \sigma_2}\x \big( \PGL_{2}(\CC)/ \Stab(\xi_3)\big) \x\widetilde{\Stab}(\xi_3).
  $$
 To compute this, note that $\widetilde{\Stab}(\xi_3)  = 
 {\tiny \left\{\left(\begin{array}{cc} \alpha & 0 \\ 0 & \alpha^{-1} \end{array}\right)\right\}} \cong \CC^*$. The action of
 $S_2$ is given by $\alpha\mapsto \alpha^{-1}$. the quotient $\CC^*/S_2\cong \CC$ via the invariant function $s=\alpha+\alpha^{-1}$.
 So $e(\CC^*/S_2)=q$, and (\ref{eqn:S2}) yields the equivariant $E$-polynomial
  \begin{equation}\label{eqn:eS2D2}
  e_{S_2}(\widetilde{\Stab}(\xi_3))= qT-N .
  \end{equation}

Also $\PGL_2(\CC)/\Stab(\xi_3)=\PGL_2(\CC)/\cD$, where $\cD$ are the diagonal matrices. 
So $\PGL_2(\CC)/\cD \cong (\PP^1)^2-\Delta$, where $\Delta=\PP^1$ is the diagonal. Hence $e(\PGL_2(\CC)/\cD)=(q+1)^2-(q+1)=q^2+q$.
For the $S_2$-quotient, we have 
$(\PGL_2(\CC)/\cD)/S_2 \cong \Sym^2(\PP^1)-\Delta$, where $\Delta=\PP^1$. Hence $e((\PGL_2(\CC)/\cD)/S_2)=
(q^2+q+1)-(q+1)=q^2$. Using (\ref{eqn:S2}), this produces the equivariant $E$-polynomial
  \begin{equation}\label{eqn:PGL2D}
  e_{S_2}(\PGL_2(\CC)/\cD)= q^2T + qN .
  \end{equation}
  
Using the equivariant $E$-polynomial of $\Delta^2_{\sigma_2 \to \sigma_2}$ in (\ref{eqn:sigma2}), and (\ref{eqn:eS2D2}) and
(\ref{eqn:PGL2D}), we have
\begin{align*}
 e_{S_2}\big(\tilde  R_{\xi_3 \to \xi_3} \big) &= ((q-n-1) T -n N) \otimes (q^2 T + qN) \otimes (qT-N) \\
 &= \left((q-n-1)(q^3-q)\right) T - n(q^3-q)N.
 \end{align*}
Therefore $e(R_{\xi_3 \to \xi_3})= (q-n-1)(q^3-q)$, the coefficient of $T$ in the expression above.

\item  We end up with $R_{\xi_3 \to \xi_1}=\tilde R_{\xi_3 \to \xi_1}/S_2$, where
 $$
 \tilde R_{\xi_3 \to \xi_1} =  \Delta^2_{\sigma_2 \to \sigma_1} \times 
 \big(\PGL_{2}(\CC)/ \Stab(\xi_3) \big)\times \widetilde{\Stab}(\xi_1).
 $$
 By Proposition \ref{prop:lemma-connected-group}, $e_{S_2}(\SL_2(\CC))=(q^3-q)T$.
 Using (\ref{eqn:sigma2}) and (\ref{eqn:PGL2D}), we get 
 \begin{align*}
e_{S_2}( \tilde R_{\xi_3 \to \xi_1}) &= ((n-1) T + (n-1) N) \otimes (q^2 T + qN) \otimes (q^3-q)T\\
 &=  (n-1)(q^3-q)(q^2+q)  T + (n-1)(q^3-q)(q^2+q)N,
 \end{align*}
 which produces $e( R_{\xi_3 \to \xi_1}) =(n-1)(q^3-q)(q^2+q)$.
\end{itemize}

Adding up all the contributions we finally get
$$
 e \big(R(H_n,\SL_2(\CC))\big) = \left( (n-1)q^2+ n q -n+5)\right)(q^3-q).
$$

\section{Rank $2$ character variety of the twisted Hopf link}\label{sec:rank2-character}

In this section, we compute the $E$-polynomial $e(\frM(H_n,G))$, for $G=\SL_2(\CC)$. First, we deal with reducible representations $(A,B)$. These
are S-equivalent to representations of the form 
 $$
  \left(\left(\begin{array}{cc} \lambda & 0 \\ 0 & \lambda^{-1} \end{array} \right) , 
\left(\begin{array}{cc} \mu & 0 \\ 0 & \mu^{-1} \end{array} \right) \right),
$$
where $(\lambda,\mu) \in (\CC^*)^2$. These are defined modulo the action 
$(\lambda,\mu)  \sim (\lambda^{-1},\mu^{-1})$. The equivariant $E$-polynomial of $\CC^*$ is
$e_{S_2}(\CC^*)=qT-N$. Hence $e_{S_2}((\CC^*)^2)=(qT-N)^2=(q^2+1)T -2q N$, and
 $$
  e(\frM^{\red}(H_n,G)) = q^2+1.
  $$
  
Now we move to the irreducible representations, which form a space $R^{\irr}(H_n,G)\subset R(H_n,G)$. As
the action is free, $\frM^{\irr}(H_n,G)=R^{\irr}(H_n,G)/\PGL_2(\CC)$.

\begin{lemma} \label{lem:reducible} Let $r\geq 2$. 
\begin{enumerate}
\item For a type $\xi$ and $G=\SL_r(\CC)$, the space $R^{\irr}(H_n,G)_{\xi \to \xi}=\emptyset$.
\item  For types $\xi'\to \xi$, if they are not diagonalizable, then $R^{\irr}(H_n,G)_{\xi' \to \xi}=\emptyset$.
\item For a type $\xi$ not corresponding to a multiple of the identity,  $R^{\irr}(H_n,G)_{\xi' \to \xi}=\emptyset$.
\end{enumerate}
  \end{lemma}
 
 \begin{proof}
  (1) Let $(A,B) \in R(H_n,G)_{\xi \to \xi}$, then $A$ and $A^n$ are of the same type. Therefore, the eigenvectors
  of $A$ and $A^n$ are the same. If $A$ and $A^n$ are a multiple of the identity, then any vector subspace
  $W\subsetneq \CC^r$ fixed by $B$ is also fixed by $A$.  Otherwise, take an eigenvalue $\lambda$ of $A$
  such that the eigenspace $W=E_\lambda(A)=E_{\lambda^n}(A^n) \subsetneq \CC^r$ is a proper subspace.
  Now $[A^n,B]=\Id$ implies that $B(W) =W$, and hence $(A,B)$ is a reducible representation.
  
  (2) Let $(A,B) \in R(H_n,G)_{\xi' \to \xi}$, and suppose that $\lambda$ is an eigenvalue of $A$ such that the
  Jordan $\lambda$-block is not diagonal. Then $E_\lambda(A)$ is a proper subspace of the $\lambda$-block,
  and hence  
   \begin{equation} \label{eqn:W}
   W=E_{\lambda^n}(A^n) =\bigoplus_{\varepsilon\in \mu_n} E_{\lambda \varepsilon}(A) \subsetneq \CC^r\, .
   \end{equation}
  Again $[A^n,B]=\Id$ implies that $B(W) =W$, and hence $(A,B)$ is a reducible representation. 

  (3) The last item is similar, taking $\lambda$ an eigenvalue of $A$, then $W$ defined in (\ref{eqn:W}) is 
 also a proper subspace of $\CC^r$.
 \end{proof}

This result implies that for $G=\SL_2(\CC)$ we only have to look at $R^{\irr}_{\xi_3 \to \xi_1}$. 
Take an irreducible $(A,B) \in R^{\irr}_{\xi_3 \to \xi_1}$. 
The matrix $A$ can be put in diagonal form with eigenvalues $(\lambda,\lambda\varepsilon)$,
$\lambda^n = \pm 1$, $\lambda^2\varepsilon=1$, $\lambda\neq \pm 1$. This is the same as to say $\lambda \in \mu_{2n}-\mu_2$.
The action of interchanging eigenvalues is an $S_2$-action free on $\mu_{2n}-\mu_2$. Therefore there 
are $n-1$ possibilities. Taking a suitable basis, then 
 \begin{equation}\label{eqn:ABirred}
(A,B)=  \left(\left(\begin{array}{cc} \lambda & 0 \\ 0 & \lambda\varepsilon \end{array} \right) , 
\left(\begin{array}{cc} a & b \\ c & d \end{array} \right) \right),
 \end{equation}

As the pair is irreducible, then $(1,0)$ and $(0,1)$ are not eigenvectors of $B$, or equivalently, $b,c\neq 0$.
This is the same as $bc\neq 0$. The action of $\cD\subset \PGL_2(\CC)$ moves $(b,c)\mapsto (\varpi^{2} b, \varpi^{-2}c)$,
so we can set $b=1$, and hence $c=1-ad$. Summing up, 
 $$
 \frM^{\irr}(H_n,G)) \cong \big (\mu_{2n}-\mu_2\big)/S_2 \times \{(a,d) \in \CC^2 \, | \, ad \neq 1\}.
 $$
 The set $ad=1$ is a hyperbola, isomorphic to $\CC^*$. Therefore
 \begin{equation}\label{eqn:Mirr2}
 e(\frM^{\irr}(H_n,G))=(n-1)(q^2-q+1).
 \end{equation}

Finally
 $$
 e(\frM(H_n,\SL_2(\CC)))=q^2+1+(n-1)(q^2-q+1).
 $$

\section{Rank $3$ representation variety of the twisted Hopf link. Combinatorial setting}\label{sec:rank3-combinatorial}

In this section and the next one, we shall compute the $E$-polynomial 
of the $\SL_{3}(\CC)$-representation variety of the twisted Hopf link $H_n$. 
For that purpose, we will analyze the combinatorial and geometric settings as described in Section \ref{sec:previous}.
We start in this section by analyzing the combinatorial setting.

Let $G=\SL_3(\CC)$. 
In this rank, up to equivalence, we have $3$ possible partitions, namely
 \begin{equation}\label{eqn:sigmai}
 \sigma_1 = \{\{1,2,3\}\}, \qquad \sigma_{2} = \{\{1,2\},\{3\}\}, \qquad \sigma_3 = \{\{1\}, \{2\}, \{3\}\}. 
 \end{equation}
They correspond to the configurations
\begin{align*}
	\Delta^3_{\sigma_1} &= \{\lambda_1 \in (\CC^*)^3\,|\, \lambda_1^3 = 1\} = \mu_3, \\
	\Delta^3_{\sigma_2} &= \{(\lambda_1, \lambda_2) \in (\CC^*)^2\,|\, \lambda_1^2\lambda_2 = 1, \lambda_1\neq \lambda_2\} 
	 = \CC^* - \Delta^3_{\sigma_1} = \CC^* - \mu_3, \\
	\Delta^3_{\sigma_3} &= \{(\lambda_1, \lambda_2, \lambda_3) \in \CC^*|\, 
	\lambda_1\lambda_2\lambda_3 = 1,  \lambda_1\neq \lambda_2\neq \lambda_3\neq \lambda_1\}.
\end{align*}

Only the stratum $\Delta^3_{\sigma_3}$ is equipped with a non-trivial action. In this case, $S_3$ acts on $\Delta^3_{\sigma_3}$ 
by permutation of $\lambda_1, \lambda_2$ and $\lambda_3$. First,
$\Delta^3_{\sigma_3}\cong (\CC^*)^2- \big( \{(\lambda_1,\lambda_1)\}\cup \{(\lambda_1,\lambda_1^{-2})\}\cup 
\{(\lambda_1^{-2},\lambda_1)\}\big)$. The curves intersect in $3$ points $\{(\varpi,\varpi)| \varpi\in \mu_3\}$, therefore
 $$
 e(\Delta^3_{\sigma_3})=(q-1)^2-3(q-1)+3\cdot 3-3= q^2-5q+10.
 $$

With respect to the refinements, for those over $\sigma_1$ we have
\begin{align*}
	e\left(\Delta^3_{\sigma_1 \to \sigma_1}\right) &= e\left(\Delta^3_{\sigma_1}\right) = 3. \\
	e\left(\Delta^3_{\sigma_2 \to \sigma_1}\right) &= e\left(\{(\lambda_1, \varepsilon) \in \CC^* \times \mu_n \,|\, 
	\lambda_1^3\varepsilon = 1, \varepsilon \neq 1\}\right)  \\
	&  = e\left(\{(\lambda_1, \varepsilon) \in \CC^* \times \mu_n \,|\, \lambda_1^3\varepsilon = 1\}\right) -  
	e\left(\{\lambda_1 \in \CC^* \,|\, \lambda_1^3 = 1\}\right)  \\
	&= e(\mu_{3n}) - e(\mu_3) = 3n-3. \\
	e\left(\Delta^3_{\sigma_3 \to \sigma_1}\right) &=  
	e\left(\{(\lambda_1, \varepsilon_1, \varepsilon_2) \in \CC^* \times (\mu_n)^2 \,|\, 
	\lambda_1^3\varepsilon_1\varepsilon_2 = 1, \varepsilon_1 \neq 1, \varepsilon_2 \neq 1, \varepsilon_1 \neq \varepsilon_2\}\right) \\
	&= e\left(\{(\lambda_1^3\varepsilon_1\varepsilon_2 = 1\}\right) 
	- e\left(\{\lambda_1^3\varepsilon_2 = 1, \varepsilon_1=1\}\right) - e\left(\{\lambda_1^3\varepsilon_1 = 1, \varepsilon_2=1\}\right) \\
	& \hspace{0.3cm} - e\left(\{\lambda_1^3\varepsilon_1^2 = 1, \varepsilon_1=\varepsilon_2\}\right) 
	+ 3\,e\left(\{\lambda_1^3 = 1, \varepsilon_1=\varepsilon_2=1\}\right) \\
	& \hspace{0.3cm} - e\left(\{\lambda_1^3 = 1, \varepsilon_1=\varepsilon_2=1\}\right) \\
	& = e(\mu_{3n} \times \mu_n) - e(\mu_{3n}) - e(\mu_{3n}) - e(\mu_{3n})+ 3\,e(\mu_3) - e(\mu_3) = 3n^2 - 9n +6.
	\end{align*}

Moreover, for those over $\sigma_2$ we have
\begin{align*}
 e\left(\Delta^3_{\sigma_2 \to \sigma_2}\right) &= e\left(\{(\lambda_1, \lambda_2) \in (\CC^*)^2 \,|\, 
	\lambda_1^2\lambda_2 = 1, \lambda_1^n \neq \lambda_2^n\}\right) \\
	& = e\left(\Delta^3_{\sigma_2}\right) - e\left(\Delta^3_{\sigma_2 \to \sigma_1}\right) = (q-4) - (3n-3) = q-3n-1. \\
 e\left(\Delta^3_{\sigma_3 \to \sigma_2}\right) &=
        e\big( \{(\lambda_1, \lambda_2, \varepsilon) \in (\CC^*)^2 \times \mu_n^* \,|\, 
	\lambda_1^2\lambda_2\varepsilon = 1, \lambda_2 \neq \lambda_1\epsilon, \forall \epsilon\in \mu_n\} \big) \\
	& = e(\{\lambda_1^2\lambda_2\varepsilon=1\}) - e(\{\epsilon \lambda_1^3\varepsilon=1,\epsilon\in \mu_n,
	\varepsilon\in \mu_n^*\})  \\
	&= e(\CC^* \times \mu_n^*) - e(\mu_{3n}\times \mu_n^*)  = (n-1)(q-3n-1).
	\end{align*}

Finally, to compute $\Delta^3_{\sigma_3\to \sigma_3}$, let us consider the three possible partitions equivalent to $\sigma_2$, namely
$$
 \sigma_2 = \{\{1,2\},\{3\}\}, \qquad \sigma_2' = \{\{1\},\{2,3\}\}, \qquad \sigma_2'' = \{\{1,3\},\{2\}\}.
$$
Notice that we have a decomposition
\begin{align}\label{eq:decom-sigma3-sigma3}
	\Delta^3_{\sigma_3} = \Delta^3_{\sigma_3\to \sigma_3} \sqcup \Delta^3_{\sigma_3\to \sigma_2}  \sqcup \Delta^3_{\sigma_3\to \sigma_2'}  \sqcup \Delta^3_{\sigma_3\to \sigma_2''}  \sqcup \Delta^3_{\sigma_3\to \sigma_1}.
\end{align}
Since $\Delta^3_{\sigma_3\to \sigma_2}  \cong \Delta^3_{\sigma_3\to \sigma_2'}  \cong \Delta^3_{\sigma_3\to \sigma_2''}$, we get
\begin{align*}
 e\left(\Delta^3_{\sigma_3 \to \sigma_3}\right) &= e\left(\Delta^3_{\sigma_3}\right) - 3e\left(\Delta^3_{\sigma_3 \to \sigma_2}\right) 
 -  e\left(\Delta^3_{\sigma_3\to \sigma_1}\right) \\
 & = q^2 - (3n + 2)q +6n^2 + 3n +1.
 \end{align*}

\subsection{Equivariant $E$-polynomials}

First, let us analyze the action of $S_3$ on $\Delta^3_{\sigma_3}$.  For the quotient $\Delta^3_{\sigma_3}/S_3$, we note that $ (\CC^*)^2/S_3$ is parametrized by $(s=\lambda_1+\lambda_2+\lambda_3, 
p=\lambda_1\lambda_2+\lambda_1\lambda_3+\lambda_2\lambda_3)\in \CC^2$. The image of
the three lines that we have to remove in $\Delta^2_{\sigma_3}$ is just one line $\CC^*$ in $\Delta^3_{\sigma_3}/S_3$, so $e(\Delta^3_{\sigma_3}/S_3)=q^2-q+1$. 
Finally, $\Delta^3_{\sigma_3}/\la \tau\ra$ is parametrized by $(u=\lambda_1+\lambda_2, 
v=\lambda_1\lambda_2) \in \CC\x\CC^*$, removing the lines $\{(\lambda_1, \lambda_1)\} \cup \{(\lambda_1,\lambda_1^{-2})\}$,
which intersect in $3$ points. Thus $e(\Delta^3_{\sigma_3}/\la \tau\ra)=q^2-q-2(q-1)+3=q^2-3q+5$. This produces:
 \begin{equation}  \label{eqn:S3otro}
  e_{S^3}(\Delta^3_{\sigma_3}) = (q^2-q+1) T+ S + (-2q+4)D. 
  \end{equation}

The configuration spaces with actions are  $\Delta^3_{\sigma_3 \to \sigma_1}$,  
$\Delta^3_{\sigma_3 \to \sigma_2}$ and  $\Delta^3_{\sigma_3 \to \sigma_3}$, which can be analyzed as follows.

\begin{itemize}
\item For the first one, we observe that the action of $S_3$ and $\tau=(1,2)$ 
on $\Delta^3_{\sigma_3 \to \sigma_1}$ are free (they interchange different eigenvalues). 
Since $\Delta^3_{\sigma_3 \to \sigma_1}$ is just a finite collection of points, we directly get that 
$e\left(\Delta^3_{\sigma_3 \to \sigma_1}/S_3\right) = (3n^2 - 9n +6)/6 = \frac12 (n^2 - 3n + 2)$ 
and $e\left(\Delta^3_{\sigma_3 \to \sigma_1}/\la\tau\ra\right) = \frac12 (3n^2 - 9n +6)$. Therefore, using (\ref{eqn:S3}) we have
 \begin{align}\label{eq:sigma3-sigma3-S3}
   e_{S_3}\left(\Delta^3_{\sigma_3 \to \sigma_1}\right) &= \frac12 \left(n^2 - 3n + 2\right) T + \frac12 \left( n^2 - 3n + 2\right)S + (n^2 - 3n + 2)D.
 \end{align}
	
\item For the second space, observe that the action of $S_3$ on $\Delta^3_{\sigma_3}$ does not restrict to 
$\Delta^3_{\sigma_3 \to \sigma_2}$. Instead, the acting subgroup is $S_{\sigma_3 \to \sigma_2} = \langle\tau \rangle < S_3$ 
with action given by $\tau \cdot (\lambda_1, \lambda_1 \varepsilon, \lambda_2) = (\lambda_1 \varepsilon, \lambda_1, \lambda_2)$,
in terms of the eigenvalues, or equivalently, $\tau \cdot (\lambda_1, \lambda_2, \varepsilon) = (\lambda_1\varepsilon, \lambda_2, \varepsilon^{-1})$
on $(\CC^*)^2\x\mu_n^*$.

Let us focus first on the natural extension of this action to $\{\lambda_1^2\lambda_2\varepsilon=1\} \cong \CC^* \times \mu_{n}^*$, 
given as $\tau \cdot (\lambda_1, \varepsilon) = (\lambda_1\varepsilon, \varepsilon^{-1})$. These are $n-1$ different 
punctured lines and the action depends on the value of $\varepsilon$. 
If $\varepsilon = -1$ (which can only happen if $n$ is even) the action is 
$\lambda_1 \mapsto -\lambda_1$ whose quotient is $\CC^*$; for $\varepsilon \neq \pm 1$, the action 
interchanges the pair of lines $\CC^* \times \{\varepsilon\}$ to $\CC^* \times \{\varepsilon^{-1}\}$ so the 
quotient is just one of them. In this way,
$$
 e((\CC^* \times \mu_n^*)/\la\tau\ra) =  \floor{\frac{n}{2}} (q-1),
 $$ 
where $\floor{x}$ is the floor function 
(the greatest integer less than or equal to $x$). 
Now we have to remove, 
$\{\epsilon \lambda_1^3\varepsilon=1, \epsilon\in \mu_n,\varepsilon\in \mu_n^*\} \cong \mu_{3n}\times \mu_n^*$, 
and the action is 
$\tau\cdot (\lambda_1, \varepsilon)=(\lambda_1\varepsilon, \varepsilon^{-1})$. If $\varepsilon \neq -1$, the action is clearly
free, and if $\varepsilon=-1$ then the action is free as well. Hence this accounts form $3n(n-1)/2$ points.
%
Therefore, putting all together we get
 \begin{align*}
 e\left(\Delta^3_{\sigma_3 \to \sigma_2} / \la\tau\ra \right) 
 & = \floor{\frac{n}{2}} (q-1) - \frac32n(n-1). 
\end{align*}

Thus, we finally find that
\begin{equation}\label{eq:sigma3-sigma2-s2}
\begin{aligned}
  e_{S_2}\left(\Delta^3_{\sigma_3 \to \sigma_2}\right) =& \left(\floor{\frac{n}{2}}(q-1) - \frac{3n(n-1)}2 \right) T \\
 &  + \left(\floor{\frac{n-1}{2}}(q-1) - \frac{3n(n-1)}2 \right)N.
\end{aligned}
 \end{equation}
 
\item To study the remaining configuration, $\Delta^3_{\sigma_3\to \sigma_3}$, observe that
regarding decomposition (\ref{eq:decom-sigma3-sigma3}) the action of $S_3$ on $\Delta^3_{\sigma_3}$ 
leaves invariant $\Delta^3_{\sigma_3\to \sigma_3}$, $\Delta^3_{\sigma_3\to \sigma_1}$, 
and $\Delta^3_{\sigma_3\to \sigma_2}  \sqcup \Delta^3_{\sigma_3\to \sigma_2'}  \sqcup \Delta^3_{\sigma_3\to \sigma_2''} $. 
For this later action, we have
$$
 \left(\Delta^3_{\sigma_3\to \sigma_2}  \sqcup \Delta^3_{\sigma_3\to \sigma_2'}  \sqcup \Delta^3_{\sigma_3\to \sigma_2''} \right)/S_3 = \Delta^3_{\sigma_3\to \sigma_2} / \langle \tau \rangle.
$$
Hence, using the previous computations we get
\begin{align*}
	e\left(\Delta^3_{\sigma_3\to \sigma_3}/S_3\right) &= e\left(\Delta^3_{\sigma_3}/S_3\right) - e\left(\Delta^3_{\sigma_3\to \sigma_2} / \langle \tau \rangle\right) - e\left(\Delta^3_{\sigma_3\to \sigma_1}/S_3\right) \\
	& =  q^{2} - q  - \floor{\frac{n}{2}} (q-1)  + n^2. 
\end{align*}

Similarly, for the action of $\tau$ we have that $\Delta^3_{\sigma_3\to \sigma_3}$, $\Delta^3_{\sigma_3\to \sigma_2} $ and $\Delta^3_{\sigma_3\to \sigma_1}$ are invariant, whereas it permutes $\Delta^3_{\sigma_3\to \sigma_2'}$ and $\Delta^3_{\sigma_3\to \sigma_2''}$. Hence,
$$
	\left(\Delta^3_{\sigma_3\to \sigma_2'}  \sqcup \Delta^3_{\sigma_3\to \sigma_2''} \right)/\langle \tau \rangle = \Delta^3_{\sigma_3\to \sigma_2'} \cong \Delta^3_{\sigma_3\to \sigma_2},
$$
and therefore
\begin{align*}
	e\left(\Delta^3_{\sigma_3\to \sigma_3}/\langle \tau \rangle\right) &= e\left(\Delta^3_{\sigma_3}/\langle \tau \rangle\right) - e\left(\Delta^3_{\sigma_3\to \sigma_2} / \langle \tau \rangle\right) - e\left(\Delta^3_{\sigma_3\to \sigma_2}\right) - e\left(\Delta^3_{\sigma_3\to \sigma_1}/\langle \tau \rangle\right) \\
	& =   q^{2} -\floor{\frac{3n+4}{2}}(q-1) +  3n^2-1.
\end{align*}

In this way, using (\ref{eqn:S3}), we get
\begin{align}\begin{split}\label{eq:sigma3-sigma3-s3}
	e_{S_3}\left(\Delta^3_{\sigma_3\to \sigma_3}\right) =& \left(q^{2} - q  - \floor{\frac{n}{2}} (q-1)  + n^2 \right) T 
	- \left(\floor{\frac{n-1}{2}} (q-1)- n^2  \right) S \\ 
	 &- \left((n+1) (q-1) - 2n^2 \right)D.
\end{split}
\end{align}

\end{itemize}

\section{Rank $3$ representation variety of the twisted Hopf link. Geometric setting}\label{sec:rank3-geometric}

From the three partitions $\sigma_1,\sigma_2,\sigma_3$ in 
(\ref{eqn:sigmai}), we can create six types (up to equivalence), namely
 \begin{align*}
 &\xi_1 = (\sigma_1, \{1,2,3\}), && \xi_2 = (\sigma_1, \{(1,2),3\}), && \xi_3 = (\sigma_1, \{(1,2,3)\}), \\
 &\xi_4=(\sigma_2, \{\{1,2\},3\}), && \xi_5=(\sigma_2,\{\{(1,2)\},3\}), && \xi_6 = (\sigma_3,\{\{1\},\{2\},\{3\} \}).
 \end{align*}
The type $\xi_j$ is given by the matrix $A_j$, $j=1,2,3,4,5,6$, where 
 \begin{itemize}
 \item $A_1=\small\left( \begin{array}{ccc} \lambda_1 & 0 & 0 \\ 0& \lambda_1 & 0 \\ 0 & 0& \lambda_1  \end{array} \right)$, with $\lambda_1\in \mu_3$.
 \item $A_2=\small\left( \begin{array}{ccc} \lambda_1 & 0 & 0 \\ 1& \lambda_1 & 0 \\ 0 & 0& \lambda_1  \end{array} \right)$, with $\lambda_1\in \mu_3$.
 \item $A_3=\small\left( \begin{array}{ccc} \lambda_1 & 0 & 0 \\ 1& \lambda_1 & 0 \\ 0 & 1& \lambda_1  \end{array} \right)$, with $\lambda_1\in \mu_3$.
 \item $A_4=\small\left( \begin{array}{ccc} \lambda_1 & 0 & 0 \\ 0& \lambda_1 & 0 \\ 0 & 0& \lambda_2  \end{array} \right)$, with $\lambda_1\neq \lambda_2$.
Here $\lambda_2=\lambda_1^{-2}$, so $\lambda_1\in \CC^*-\mu_3$.
 \item $A_5=\small\left( \begin{array}{ccc} \lambda_1 & 0 & 0 \\ 1& \lambda_1 & 0 \\ 0 & 0& \lambda_2  \end{array} \right)$, with 
 $\lambda_2=\lambda_1^{-2}$, and $\lambda_1\in \CC^*-\mu_3$.
 \item $A_6=\small\left( \begin{array}{ccc} \lambda_1 & 0 & 0 \\ 0& \lambda_2 & 0 \\ 0 & 0& \lambda_3  \end{array} \right)$, with $\lambda_i\neq \lambda_j$
for $i\neq j$, $\lambda_1\lambda_2\lambda_3=1$.
\end{itemize}
Only in the case $\xi_6$ we have an action of the group $S_3$ given by the permutation of the eigenvalues $\lambda_i$.

Denote $R_{\xi_i\to \xi_j}=R(H_n,\SL_3(\CC))_{\xi_i\to \xi_j}$. 
Taking into account the possible refinement relations, we have
$$
 R(H_n,\SL_3(\CC)) = R_{\xi_4 \to \xi_1} \sqcup R_{\xi_6 \to \xi_4}
 \sqcup R_{\xi_6 \to \xi_1}\sqcup R_{\xi_5 \to \xi_2} \sqcup \bigsqcup_{1\leq i\leq 6} R_{\xi_i \to \xi_i}.
$$
Using Proposition \ref{prop:lem:RG}, we have 
 \begin{equation}\label{eqn:ij}
 R_{\xi_i \to \xi_j} \cong  
 \left.\left(\cA_{\xi_i \to \xi_j} \times \big(\PGL_{3}(\CC)/ \Stab(\xi_i)\big) \times \widetilde{\Stab}(\xi_j) \right) \right/ S_{\xi_i \to \xi_j} .
 \end{equation}
 
\subsection{The stabilizers} 
We start by studying $\PGL_{3}(\CC)/\Stab(\xi_i)$ and $\widetilde{\Stab}(\xi_j)$. Recall that $e(\PGL_3(\CC))=e(\SL_3(\CC))=
(q^3-1)(q^3-q)q^2$.

\begin{itemize}
 \item $\widetilde{\Stab}(\xi_1)=\SL_3(\CC)$.\vspace{2mm}
 
 \item $\widetilde{\Stab}(\xi_2)=\left\{{\small \left( \begin{array}{ccc} \beta_1 & 0 & 0 \\ a& \beta_1 & b \\ c & 0& \beta_1^{-2}  \end{array} \right)}\Big|\beta_1\in \CC^*\right\}
 \cong \CC^*\x \CC^3$. 
 
 \noindent Then  $e(\widetilde{\Stab}(\xi_2))=(q-1)q^3$, and $e(\PGL_3(\CC)/\Stab(\xi_2))= (q^3-1)(q+1)$. \vspace{2mm}

 \item $\widetilde{\Stab}(\xi_3)=\left\{{\small \left( \begin{array}{ccc} \beta_1 & 0 & 0 \\ a& \beta_1 & 0 \\ b & a& \beta_1  \end{array} \right)}\Big|\beta_1\in \mu_3\right\}
 \cong \mu_3 \x\CC^2$. 
 
 \noindent Thus  $e(\widetilde{\Stab}(\xi_3))=3q^2$, $\Stab(\xi_3)=q^2$, and $e(\PGL_3(\CC)/\Stab(\xi_3))=(q^3-1)(q^2-1)q$.\vspace{2mm}

 \item $\widetilde{\Stab}(\xi_4)=\left\{{\small \left( \begin{array}{c|c} B' & 0 \\ \hline 0& (\det B')^{-1}  \end{array} \right)}\Big| B'\in \GL_2(\CC)\right\}$.
 
 \noindent Thus  $e(\widetilde{\Stab}(\xi_4))=(q^2-1)(q^2-q)$, and $e(\PGL_3(\CC)/\Stab(\xi_4))= (q^2+q+1)q^2$.\vspace{2mm}
 
 \item $\widetilde{\Stab}(\xi_5)=\left\{{\small \left( \begin{array}{ccc} \beta_1 & 0 & 0 \\ a& \beta_1 & 0 \\ 0 & 0& \beta_1^{-2} \end{array} \right)}\Big
  |\beta_1\in \CC^*\right\} \cong \CC^*\x \CC$. 
  
  \noindent Thus  $e(\widetilde{\Stab}(\xi_5))=q(q-1)$, and $e(\PGL_3(\CC)/\Stab(\xi_5))=(q^3-1)(q^3+q^2)$. \vspace{2mm}

 \item $\widetilde{\Stab}(\xi_6)=\left\{{\small \left( \begin{array}{ccc} \beta_1 & 0 & 0 \\ 0& \beta_2 & 0 \\ 0 & 0& \beta_3  \end{array} \right)} \Big| \beta_1\beta_2\beta_3=1\right\}
 =\cD \cong \Delta^3_{\sigma_3}= (\CC^*)^2$.
 
\noindent  Then  $e(\widetilde{\Stab}(\xi_6))=(q-1)^2$, and $e(\PGL_3(\CC)/\Stab(\xi_6))=(q^2+q+1)(q+1)q^3$.
\end{itemize}

Only in the case of $\xi_6$ there is an action of $S_3$. Let us give the equivariant $E$-polynomials.
The quotient $\cD/S_3$ is parametrized by 
$(s=\beta_1+\beta_2+\beta_3, p=\beta_1\beta_2+\beta_1\beta_3+\beta_2\beta_3)\in\CC^2$, hence
$e(\cD/S_3)=q^2$. Also $\Delta^3_{\sigma_3}/\la \tau\ra= (\CC^*/\ZZ_2)\x \CC^* \cong \CC\x \CC^*$,  parametrized
by $u=\beta_1+\beta_2$, $v=\beta_1\beta_2$, $\beta_3=v^{-1}$, 
and hence $e(\cD/\la \tau\ra)=q^2-q$. All together (\ref{eqn:S3}) yields
  \begin{equation}\label{eqn:D}
  e_{S_3}(\cD)= q^2 T+ S -q D.
  \end{equation}
  
The quotient $\PGL_3(\CC)/\cD$ is parametrized by the column vectors of the matrix in $\GL_3(\CC)$ up to scalar.
This means that 
 $$
\PGL_3(\CC)/\cD=\{([v_1],[v_2],[v_3]) \in (\PP^2)^3 \text{ linearly independent}\}= (\PP^2)^3- \cL,
$$
where $\cL$ is the subspace of three linearly dependent points of $\PP^2$. This is $\cL=\cL^0 \sqcup \cL^1$, according to 
whether they are coincident, or they span a line. In the first case $\cL^0\cong \PP^2$, and
in the second, $\cL^1 \cong \Gr(\PP^1,\PP^2) \x ((\PP^1)^3-\cL^0)$, given by choosing a line and three not-all-equal points of it.
The Grassmannian of lines is the dual projective plane, so 
 \begin{align*}
 e( \PGL_3(\CC)/\cD) &=   (q^2+q+1)^3- \Big((q^2+q+1) + (q^2+q+1)((q+1)^3-(q+1)) \Big) \\
 &= q^6+2q^5+2q^4+q^3.
\end{align*}
This agrees with $e(\PGL_3(\CC)/\cD)=(q^3-1)(q^3-q)q^2/(q-1)^2$.

For the quotient, $(\PGL_3(\CC)/\cD)/S_3 = \Sym^3(\PP^2)- \cL$,
where now $\cL=\PP^2 \sqcup (\Gr(\PP^1, \PP^2) \x (\Sym^3(\PP^1)-\PP^1))$. Then
\begin{align*}
 e((\PGL_3(\CC)/\cD)/S_3) =&\,  (q^6+q^5+2q^4+2q^3+2q^2+q+1) \\ &- \Big((q^2+q+1) + (q^2+q+1)((q^3+q^2+q+1)-(q+1)) \Big) \\
 =&\,  q^6.
\end{align*}
Finally, $(\PGL_3(\CC)/\cD)/\la \tau \ra = \Sym^2(\PP^2) \x \PP^2 - \cL$,
where  $\cL=\PP^2 \sqcup ( \Gr(\PP^1, \PP^2) \x (\Sym^2(\PP^1)\x \PP^1-\PP^1))$. Therefore
\begin{align*}
 e((\PGL_3(\CC)/\cD)/\la\tau\ra) =&\,    (q^4+q^3+2q^2+q+1)(q^2+q+1) \\ 
&- \Big((q^2+q+1) + (q^2+q+1)((q^2+q+1)(q+1)-(q+1)) \Big) \\
 =&\,  q^6+q^5+q^4.
\end{align*}

All together, understanding $S_2 = \langle\tau\rangle$, this gives
 \begin{align}\label{eqn:PGLD}
 e_{S_3}(\PGL_3(\CC)/\cD) &= q^6 T+q^3 S + (q^5+q^4)D, \\
 e_{S_2}(\PGL_3(\CC)/\cD) &= (q^6+q^5+q^4) T+ (q^5 + q^4 + q^3)N. \label{eqn:PGLD-S2}
 \end{align}

Finally, for the action of $S_3$ on $\widetilde{\Stab}(\xi_1)= \SL_3(\CC)$ and the action of $S_2$ on $\widetilde{\Stab}(\xi_4) \cong \GL_2(\CC)$,
by Proposition \ref{prop:lemma-connected-group} we have that
\begin{align} \label{eq:e-pols-equiv-SL-GL}
	 e_{S_3}(\SL_3(\CC)) = (q^8 - q^6 - q^5 + q^3)T, \qquad e_{S_2}(\GL_2(\CC)) = (q^4-q^3-q^2+q) T .
\end{align}

\begin{remark}
There is a locally trivial fibration $\cD \to \SL_3(\CC)  \to \PGL_3(\CC)/\cD$. This is an equivariant fibration, which can be seen 
as follows: the base $B=\PGL_3(\CC)/\cD$ parametrizes subsets of $3$ non-collinear points in $\PP^2$, which is an open subset
of $\Sym^3(\PP^2)$. Given one triplet, 
there is a line $\ell\subset \PP^2$ missing it, hence it lies in $\Sym^3 (\PP^2- \ell)$. The fibration over $B\cap \Sym^3 (\PP^2- \ell)$ is trivial and $S_3$-invariant, which shows the claim.
In particular, it holds $e_{S_3}(\SL_3(\CC))= e_{S_3}(\cD) \otimes e_{S_3}(\PGL_3(\CC)/\cD)$. 
This can also be checked from (\ref{eqn:D}) and (\ref{eqn:PGLD}). Similar 
observations apply for $\GL_2(\CC)$.
\end{remark}

\subsection{Adding up all contributions}
Now we move to the computation of the $E$-polynomials of the strata (\ref{eqn:ij}).
For $i=1,\ldots, 5$, we have 
 $$ 
 R_{\xi_i \to \xi_i} \cong  \Delta^3_{\sigma_{(i)} \to \sigma_{(i)}} 
 \times \big(\PGL_{3}(\CC)/ \Stab(\xi_i)\big) \times \widetilde{\Stab}(\xi_i) ,
 $$
using Lemma \ref{lem:A}, where $\sigma_{(i)}$ denotes the partition associated to $\xi_i$.
Therefore 
 \begin{itemize}
 \item $e(R_{\xi_1 \to \xi_1} )=e( \Delta^3_{\sigma_1 \to \sigma_1}) e(\PGL_3(\CC))= 3(q^3-1)(q^3-q)q^2$.
 \item $e(R_{\xi_2\to \xi_2} )=e( \Delta^3_{\sigma_1 \to \sigma_1}) e(\PGL_3(\CC)) =3(q^3-1)(q^3-q)q^2$.
 \item $e(R_{\xi_3\to \xi_3} )=e( \Delta^3_{\sigma_1 \to \sigma_1}) 3\,e(\PGL_3(\CC))= 9(q^3-1)(q^3-q)q^2$.
 \item $e(R_{\xi_4 \to \xi_4} )=e( \Delta^3_{\sigma_2 \to \sigma_2}) e(\PGL_3(\CC)) =(q-3n-1)(q^3-1)(q^3-q)q^2$.
 \item $e(R_{\xi_5 \to \xi_5} )=e( \Delta^3_{\sigma_2 \to \sigma_2}) e(\PGL_3(\CC)) =(q-3n-1)(q^3-1)(q^3-q)q^2$.
 \end{itemize}

The remaining five strata are analyzed one by one:
\begin{itemize}

\item $R_{\xi_5 \to \xi_2} \cong  \Delta^3_{\sigma_2 \to \sigma_1} 
 \times \big(\PGL_{3}(\CC)/ \Stab(\xi_5)\big) \times \widetilde{\Stab}(\xi_2)$,
hence 
  \begin{align*}
 e(R_{\xi_5 \to \xi_2}) &=e(  \Delta^3_{\sigma_2 \to \sigma_1}) e\big(\PGL_{3}(\CC)/ \Stab(\xi_5)\big) e(\widetilde{\Stab}(\xi_2)) \\
  &=(3n-3) (q^3-1)(q^3+q^2)(q-1)q^3.
  \end{align*} 

\item $R_{\xi_4 \to \xi_1} \cong  \Delta^3_{\sigma_2 \to \sigma_1} 
 \times \big(\PGL_{3}(\CC)/ \Stab(\xi_4)\big) \times \widetilde{\Stab}(\xi_1)$, therefore
  \begin{align*}
 e(R_{\xi_4 \to \xi_1}) &=e(  \Delta^3_{\sigma_2 \to \sigma_1}) e\big(\PGL_{3}(\CC)/ \Stab(\xi_4)\big) e(\widetilde{\Stab}(\xi_1)) \\
  &=(3n-3) (q^2+q+1)q^2(q^3-1)(q^2-1)q^3.
  \end{align*}

\item $R_{\xi_6 \to \xi_1} \cong \tilde R_{\xi_6 \to \xi_1}/S_3$, where
 $\tilde R_{\xi_6 \to \xi_1} \cong  \Delta^3_{\sigma_3 \to \sigma_1}  \times \big(\PGL_{3}(\CC)/ \Stab(\xi_6)\big) \times \widetilde{\Stab}(\xi_1) =  \Delta^3_{\sigma_3 \to \sigma_1}  \times \big(\PGL_{3}(\CC)/ \cD\big) \times \SL_3(\CC)$.
By (\ref{eq:sigma3-sigma3-S3}), (\ref{eqn:PGLD}) and (\ref{eq:e-pols-equiv-SL-GL}), we have 
  \begin{align*}
 e_{S_3}(\tilde R_{\xi_6 \to \xi_1}) &= 
 \left( \frac{n^2 - 3n + 2}2 T + \frac{n^2 - 3n + 2}2S + (n^2 - 3n + 2)D \right) \\ & \qquad
 \ox ( q^6 T+q^3 S + (q^5+q^4)D ) \ox (q^8 - q^6 - q^5 + q^3)T .
  \end{align*} 
 Taking the $T$-component,
  \begin{align*}
 e(R_{\xi_6 \to \xi_1}) &= \frac{n^2-3n+2}{2} (q^3-1)^2 (q+1)^2 q^6 .
  \end{align*} 

\item $R_{\xi_6 \to \xi_4} \cong \tilde R_{\xi_6 \to \xi_4}/S_2$, where
 $R_{\xi_6 \to \xi_4} \cong  \Delta^3_{\sigma_3 \to \sigma_2}  \times \big(\PGL_{3}(\CC)/ \Stab(\xi_6)\big) \times \widetilde{\Stab}(\xi_4)$.
By (\ref{eq:sigma3-sigma2-s2}), (\ref{eqn:PGLD-S2}) and (\ref{eq:e-pols-equiv-SL-GL}) we have
  \begin{align*}
 e_{S_2}(\tilde R_{\xi_6 \to \xi_4}) &=  
 \left(\left(\floor{\frac{n}{2}}(q-1) - \frac{3n(n-1)}2 \right) T 
   + \left(\floor{\frac{n-1}{2}}(q-1) - \frac{3n(n-1)}2 \right)N\right)\\
  & \qquad \ox \left((q^6+q^5+q^4) T+ (q^5 + q^4 + q^3)N\right) \ox (q^4-q^3-q^2+q) T.
  \end{align*} 
  Taking the $T$-component,
  \begin{align*}
 e(R_{\xi_6 \to \xi_4}) = & \, q^4 (q^3-1)(q^2-1) 
 \Big( \floor{\frac{n}2}(q-1)^2 -\frac{3n^2-5n+2}2 (q-1) -3n(n-1)\Big).
%
%
\end{align*}

\item  $R_{\xi_6 \to \xi_6} =\tilde R_{\xi_6 \to \xi_6}/S_3$, where 
 $\tilde R_{\xi_6\to \xi_6} \cong  \Delta^3_{\sigma_3 \to \sigma_3} \times \big(\PGL_{3}(\CC)/ \cD\big) \times \cD$.
By (\ref{eq:sigma3-sigma3-s3}), (\ref{eqn:PGLD}) and (\ref{eqn:D}) we have
  \begin{align*}
 e_{S_3}( &\tilde R_{\xi_6 \to \xi_6}) = \Big(
  \left(q^{2} - q  - \floor{\frac{n}{2}} (q-1)  + n^2 \right) T 
	- \left(\floor{\frac{n-1}{2}} (q-1)- n^2  \right) S  \\ 
	 &- \left((n+1) (q-1) - 2n^2 \right)D \Big) \ox \left(q^6 T+q^3 S + (q^5+q^4)D\right) 
	\ox \left(q^2 T+ S -q D\right).
	  \end{align*}
   
  Taking the $T$-component,
  \begin{align*}
 e(R_{\xi_6 \to \xi_6}) &=   \left(q^{2} - q  - \floor{\frac{n}{2}} (q-1)  + n^2 \right) (q^8-q^6-q^5+q^3)
\end{align*}
\end{itemize}

Adding up all the contributions, we finally get
\begin{align*}
e\big(R(H_n, & \SL_{3}(\CC))\big) =\,  (q^3-1)(q^2-1)q^2 \Big( \floor{\frac{n}{2}} (q^2-q)(q^2-q-1) \\
 & +\frac12n^2 (q^7+2q^6+2q^5+q^4-3q^3-3q^2+2q) 
 -\frac12n (3q^7+6q^6-3q^4 -17q^3 \\ & \qquad -q^2+12q)  
 + q^7+2q^6-q^5-2q^4-6q^3+2q^2+13q\Big).
%
%
%
\end{align*}

\section{Rank $3$ character variety of the twisted Hopf link}\label{sec:rank3-character}

We end up with the computation of $e(\frM(H_n,G))$, for 
$G=\SL_3(\CC)$. First, we deal with reducible representations $(A,B)$. The ones of type $(1,1,1)$ are 
the direct sums of three one-dimensional representations. This means that
 $$
 (A,B)=  \left(\left(\begin{array}{ccc} \lambda_1 &0 & 0 \\ 0 & \lambda_2 &0 \\ 0 &0&\lambda_3 \end{array} \right) , 
 \left(\begin{array}{ccc} \mu_1 &0 & 0 \\ 0 & \mu_2 &0 \\ 0 &0&\mu_3 \end{array} \right)\right) , 
$$
which is parametrized by $(\cD)^2/S_3$. Using (\ref{eqn:D}), we take the $T$-component of 
$e_{S_3}((\cD)^2)=e_{S_3}(\cD)^2=(q^2T+S-qD)^2$, which is 
 $$
  e(\frM^{\red}_{1,1,1}(H_n,G))= e((\cD)^2/S_3)= q^4+ q^2 + 1\,.
  $$
  
Next, we consider the reducible representations of type $(2,1)$. Then 
 $$
 (A,B)=  \left(\left(\begin{array}{c|c} A_1 &0  \\ \hline 0 & \lambda_1 \end{array} \right) , 
 \left(\begin{array}{c|c} B_1 &0  \\ \hline 0 & \mu_1 \end{array} \right) \right) , 
$$
where $\lambda_1=(\det A_1)^{-1}$, $\mu_1=(\det B_1)^{-1}$, and $(A_1,B_1)$ is an irreducible
$\GL_2(\CC)$-represen\-ta\-tion. The computation is similar to the case of $\SL_2(\CC)$ in Section \ref{sec:rank2-character}.
Lemma \ref{lem:reducible} applies here, and we only have to see at the reductions $\xi_3\to \xi_1$. Therefore
$(A_1,B_1)$ can be put on the form (\ref{eqn:ABirred}), where $\lambda \in \CC^*$, $\varepsilon \in \mu_n^*$, 
 $bc\neq 0$, $ad-bc\neq 0$. Here, we find two options.
 \begin{enumerate}
 	\item  If $\varepsilon \neq -1$ (which happens always if $n$ is odd), then 
 the $S_2$-action sends $\varepsilon\mapsto \varepsilon^{-1}$. The quotient of one of such
 sets is parametrized by $(\GL_2(\CC)-\{bc=0\})/\CC^*$, whose $E$-polynomial is $(q^2-q+1)(q-1)=q^3-2q^2+2q-1$.
 	\item  If $\varepsilon=-1$ then we have to quotient by the swap of the eigenvalues, which yields the space
 $((\GL_2(\CC)-\{bc=0\})/\CC^*)/S_2$. Let us consider the fibration
 $$
 	F=\{(a,d) \,|\, ad \neq bc+1\} \to \GL_2(\CC)-\{bc=0\} \to B=\{(b,c) \,|\, bc\neq 0\}.
$$
 Note that the action of $\CC^*$ is only on $B$, $(b,c)\mapsto (b\varpi^2, c\varpi^{-2})$. The quotient
 is $B/\CC^* \cong \CC^*$ with trivial $S_2$-action, whence $e_{S^2}(B/\CC^*)=(q-1)T$.
 For the fiber, $S_2$ swaps $(a,d)$, hence the quotient is parametrized by $s=a+d$, $p=ad \neq bc+1$, so
 $e(F/S_2)=q(q-1)$, and thus $e_{S^2}(F)=(q^2-q)T+N$. Therefore we get
 \begin{align*}
 	e_{S_2}((\GL_2(\CC)-\{bc=0\})/\CC^*) &= e_{S_2}(B/\CC^*) e_{S_2}(F) \\
	&=(q^3-2q^2+q)T+ (q-1)N.
 \end{align*}
 \end{enumerate}
 
 All together, taking into account the contribution of $\lambda \in \CC^*$, we have
 \begin{align*}
  e(\frM^{\red}_{2,1}(H_n,G)) &= (q-1)\bigg(\floor{\frac{n-1}2} (q^3-2q^2+2q-1) \\
  & \qquad + \Big(n-1-2\floor{\frac{n-1}2}\Big)(q^3-2q^2+q)\bigg) \\
   &= (n-1)(q^3-2q^2+q)(q-1) - \floor{\frac{n-1}2} (q^3 - 2q^2 + 1)(q-1) \,.
\end{align*}

Now we move to $\frM^{\irr}(H_n,G)$. 
For the irreducible representations, Lemma \ref{lem:reducible} implies that the
only non-empty strata are $R^{\irr}_{\xi_4\to\xi_1}$ and $R^{\irr}_{\xi_6\to\xi_1}$.
We start with  $(A,B)\in R^{\irr}_{\xi_6\to\xi_1}$. Choosing a suitable basis,
 $$
 (A,B)=  \left(\left(\begin{array}{ccc} \lambda_1 &0 & 0 \\ 0 & \lambda_2 &0 \\ 0 &0&\lambda_3 \end{array} \right) , 
\left(\begin{array}{ccc} a & b & c \\ d & e & f \\ g & h & i \end{array} \right) \right),
$$
modulo the action of $\cD=(\CC^*)^2$.
As $\lambda_1^n=\lambda_2^n=\lambda_3^n=\varpi$, $\varpi^3=1$, $\lambda_i\neq \lambda_j$ for $i\neq j$,
the count of matrices is given by $\Delta^3_{\sigma_3\to \sigma_1}/S_3$, i.e.\ $\frac12 (n^2-3n+2)$ points.
In order for $(A,B)$ to be irreducible, they cannot leave invariant a line (that is, no column of $B$ is the
coordinate vector) or a plane (that is, no row of $B$ is the coordinate vector). We count the contribution:
 \begin{itemize}
 \item $d\neq 0$ and $g\neq 0$. Using the action of $\cD$, we arrange $d=1, g=1$. The space of such matrices in
 $\SL_3(\CC)$ has $E$-polynomial $q(q^3-q)q^2$. Now we have to remove 
 $U_1=\{b=c=0\}$, $U_2=\{b=h=0\}$, $U_3=\{c=f=0\}$. Denote $U_{ij}=U_i\cap U_j$. Note that $U_{23}=U_{123}$.
 The contribution is: 
\begin{align*} 
 e(U_1\cup U_2 &\cup U_3 ) =  e(U_1)+ e(U_2)+ e(U_3 ) - e(U_{12})- e(U_{13}) \\ 
 &= (q^2-1)(q^2-q) + q^2(q^2-q) + q^2(q^2-q) - q(q-1)^2- q(q-1)^2 \\
 &= 3 q^{4} - 5  q^{3} + 3 q^{2} - q.
 \end{align*}
Hence the $E$-polynomial of this stratum is $q^{6} - 4 q^{4} + 5 q^{3} - 3 q^{2} + q$.
 
 \item  $d\neq 0$ and $g=0$. It must be $h\neq 0$. We can arrange $d=1$, $h=1$. 
Then $B$ has determinant $c+aei-bi-af=1$, so $c$ is fixed.
Therefore the space of such matrices in $\SL_3(\CC)$ has $E$-polynomial $q^5$.
Now we remove $U_1=\{b=c=0\}$, $U_2=\{c=0,f=0\}$.
 The contribution is: 
\begin{align*} 
 e(U_1\cup U_2 ) &=  e(U_1)+ e(U_2) - e(U_{12}) \\&= q(q^2-q) + q(q^2-q)-(q-1)^2 \\  &=2q^3-3q^2+2q-1.
  \end{align*}
Thus the $E$-polynomial of this stratum is $q^5-2q^3+3q^2-2q+1$.

 \item  $d=0$, $g\neq 0$. This is analogous to the previous one. It has $E$-polynomial $q^5-2q^3+3q^2-2q+1$.
\end{itemize}
Adding up,
  $$
  e(\frM^{\irr}_{\xi_6\to\xi_1} )= \frac12 (n^2-3n+2)(q^{6} + 2 q^{5} - 4 q^{4} + q^{3} + 3  q^{2} - 3 q + 2).
  $$

We end up with  $(A,B)\in R^{\irr}_{\xi_4\to\xi_1}$. Choosing a suitable basis,
 $$
 (A,B)=  \left(\left(\begin{array}{ccc} \lambda &0 & 0 \\ 0 & \lambda &0 \\ 0 &0&\lambda\varepsilon \end{array} \right) , 
 \left(\begin{array}{c|c} B_1 & \begin{array}{c} c \\ f \end{array} \\ \hline g \,\,\,  h & i \end{array} \right) \right),
$$
where $\varepsilon=\lambda^{-3}$, $\lambda\in \mu_{3n}- \mu_3$. This space is 
modulo the action of $\PGL_2(\CC)\times \CC^*$.
The action of $\PGL_2(\CC)$ conjugates $B_1$, therefore we can put it in Jordan form. There are two options:
 \begin{itemize}
 \item $B_1$ is diagonalizable. Therefore we can put 
 $$
B=\left(\begin{array}{ccc} a & 0 & c \\ 0 & e & f \\ g & h & i \end{array} \right) .
$$
It must be $g\neq 0$, $h\neq 0$, $c\neq 0$ and $f\neq 0$. With the residual action of $\cD=\CC^*\x \CC^*$, we 
can arrange $g=1$, $h=1$. No eigenvector of $A$ of the form $(x,y,0)$ should be eigenvector of $B$, which translates into
$a\neq e$. Also, no invariant plane of the form $\la (x,y,0),(0,0,1)\ra$ should be invariant for $B$, which also means $a\neq e$.
 Now we distinguish two cases:
 \begin{enumerate}
 \item $a,e\neq 0$. Then the condition $\det B=1$ says that $i$ is given in terms of $(a,e,c,f) \in ((\CC^*)^2-\Delta)\x (\CC^*)^2$. There is an
 action of $S_2$ swapping eigenvalues of $B_1$, that is $(a,e,c,f) \mapsto (e,a, f,c)$. 
 The equivariant $E$-polynomials are: $e_{S_2}((\CC^*)^2-\Delta) =(q^2-2q+1)T-(q-1)N$,
 $e_{S_2}((\CC^*)^2=(q^2-q)T-(q-1)N$. This gives the final $E$-polynomial
 $(q^2-q)(q^2-2q+1)+(q-1)^2= q^4-3q^3+4q^2-3q+1$.
  \item $a=0$, $e\neq 0$ (and there is no swapping of eigenvalues now). Then the parameters are $(c,f,i)\in (\CC^*)^2\x \CC$, 
  and $e=c^{-1}$. The $E$-polynomial is $(q-1)^2q$.
  \end{enumerate}
  
  \item $B_1$ is not diagonalizable. Therefore we can put
 $$
B=\left(\begin{array}{ccc} a & 0 & c \\ 1 & a & f \\ g & h & i \end{array} \right) .
$$
 There is a residual action of  $\tiny \left(\begin{array}{ccc} 1 & 0 & 0 \\ x & 1 & 0 \\ 0 & 0 & y \end{array} \right)$. 
 As it must be $h\neq 0$, we can arrange $h=1$, and also $g=0$. The irreducibility means that $(0,1,0),(0,0,1)$ are not eigenvectors of
 $B$, and $\la (1,0,0),(0,1,0)\ra$ and $\la (0,1,0), (0,0,1)\ra$ are not invariant planes of $B$. This translates into $c\neq 0$.
 The determinant condition is $\det B= a^2 i+c-af =1$, so $c$ is determined, and the space is $\{(a,i,f) | \, a^2i-af\neq 1\}$.
 For $a\neq 0$, this is $\CC^2- \CC$; and for $a=0$, it is $\CC^2$. So the $E$-polynomial is
 $(q-1)(q^2-q)+q^2=q^3-q^2+q$.
\end{itemize} 
 
Adding up, this amounts to $q^4-q^3+q^2-q+1$, and taking into account the possible values of $\lambda$ we get
  $$
  e(\frM^{\irr}_{\xi_4\to\xi_1} )=(3n-3) (q^4-q^3+q^2-q+1).
   $$

Putting all together we finally get
\begin{align*}
e(\frM(H_n, \SL_3(\CC))) &=  q^4 + q^2 + 1 + \frac{1}{2}(n^2 - 3n + 2)\left(q^6 + 2q^5 - 4q^4 + q^3 + 3q^2 - 3q + 2\right) \\
& \qquad + 3(n - 1)(q^4 - q^3 + q^2 - q + 1) + (n - 1)(q-1)\left(q^3 - 2q^2 + q \right) \\
& \qquad -  \floor{\frac{n-1}2} (q^3 - 2q^2 + 1)(q-1).
\end{align*}

\begin{remark}
There is an error in the calculation of the $E$-polynomial of $\frM^{\red}_{2,1}(H_n,G)$ in the published version of this manuscript. In the case $\varepsilon=-1$, the equivariant $E$-polynomial of $B/\CC^*$ should be $e_{S^2}(B/\CC^*)=(q-1)T$, and not $e_{S^2}(B/\CC^*)=qT-N$ as computed in previous versions. The current version of the manuscript corrects this mistake and recomputes the final $E$-polynomial.
\end{remark}

\subsection*{Data availability statement}

This work has no associate data.

\end{document}